\documentclass[reqno,11pt]{amsart}
\usepackage{amssymb,latexsym}
\usepackage{amsmath}
\usepackage{amsthm}
\usepackage{color}
\usepackage{graphicx}
\usepackage{hyperref}
\usepackage{titletoc}
\numberwithin{equation}{section}
\newtheorem{theorem}{Theorem}[section]
\newtheorem{proposition}[theorem]{Proposition}
\newtheorem{lemma}[theorem]{Lemma}

\theoremstyle{definition}

\newtheorem{remark}{Remark}

\def\B{\mathbb{R}^2}
\def\BB{B_{R/\varepsilon}}
\def\Om{\Omega}

\def\e{\varepsilon}
\def\s{\sigma}
\def\q{\mathbf {q}}
\def\p{\mathbf {p}}

\usepackage{pstricks}
\usepackage{color}

\newrgbcolor{LemonChiffon}{1. 0.98 0.8}
\newrgbcolor{LightBlue}{0.68 0.85 0.9}
\newrgbcolor{darkgreen}{0.0 0.45 0.0}
\newrgbcolor{brightgreen}{0.0 0.85 0.0}
\newrgbcolor{verylightgray}{0.9 0.9 0.9}
\newrgbcolor{gray}{0.4 0.4 0.4}
\newrgbcolor{Red}{1 0 0}
\newrgbcolor{Blue}{0 0 1}

\begin{document}
\begin{abstract}
We consider the Gierer-Meinhardt system with small inhibitor diffusivity, very small activator diffusivity and a precursor inhomogeneity.
For any given positive integer $k$
we construct a spike cluster consisting of $k$ spikes which all approach the same nondegenerate local minimum point of the precursor inhomogeneity. We show that this spike cluster can be linearly stable.
In particular, we show  the existence of spike clusters for spikes located at the vertices of a polgyon with or without centre.
Further, the cluster without centre is stable for up to three spikes, whereas the cluster with centre is stable for up to six spikes.

The main idea underpinning
these stable spike clusters
is the following: due to the small inhibitor diffusivity
the  interaction between spikes
is repulsive, and the spikes are attracted towards
the local minimum point of the precursor inhomogeneity.
Combining these two effects can lead to an equilibrium of spike positions within the cluster such that the cluster is linearly stable.
\end{abstract}

\title[Precursor Gierer-Meinhardt system in $\mathbb{R}^2$]{Stable spike clusters for the precursor Gierer-Meinhardt system in $\mathbb{R}^2$}

\author{ Juncheng Wei}
\thanks{ Juncheng ~Wei,~Department of Mathematics, University of British Columbia,
Vancouver, BC V6T 1Z2, Canada, email: jcwei@math.ubc.ca}
\author{Matthias Winter}
\thanks{Corresponding author: Matthias Winter, ~Department of Mathematics, Brunel University London, Uxbridge UB8 3PH, United Kingdom, Tel. +441895 267179, email: matthias.winter@brunel.ac.uk}
\author{ Wen Yang}
\thanks{Wen Yang, Center for Advanced Study in Theoretical Sciences (CASTS), National
Taiwan University, Taipei 10617, Taiwan, email: math.yangwen@gmail.com}
\date{}

\maketitle


{{\it Acknowledgements:}
The  research  of J.~Wei is partially supported by Natural Sciences and Engineering Research Council of Canada.
MW thanks the Department of Mathematics at the University of British Columbia for their kind hospitality.
}

\section{Introduction}
In 1952, Turing \cite{t} studied how pattern formation could start from a state without patterns. He explained the onset of pattern formation by a combination of two properties of the system:

(i) presence of spatially varying instabilities

(ii) absence of spatially homogeneous instabilities.

Since Turing's pioneering work many models have been proposed and studied to explore the so-called {\em Turing diffusion-driven instability} in reaction-diffusion systems to understand biological pattern formation. One of the most popular of these models is the Gierer-Meinhardt system \cite{gm,m1,m2}, which in two dimensions with a precursor-inhomogeneity and two small diffusivities can be stated as follows:
\begin{align}
\label{1.1}
\begin{cases}
A_t=\e^2\Delta A-\mu A+\frac{A^2}{H}, \quad &\mbox{in }\Om, \\[2mm]
\tau H_t=D\Delta H-H+A^2, \quad &\mbox{in }\Om,\\[2mm]
\frac{\partial A}{\partial \nu}=\frac{\partial H}{\partial \nu}{=0}, \quad &\mbox{on }\partial \Om,
\end{cases}
\end{align}
where $\e$ and $D$ are positive constants  satisfying $0<\e^2\ll D\ll\frac{1}{\log\frac{\sqrt{D}}{\e}}\ll1$ and $\tau$ is a nonnegative constant which is independent of $\e.$ Further, $\Om$ is a smooth bounded domain and $\mu$ is a sufficiently smooth positive function of the spatial variable in this domain. In this paper we assume that $\Om$ is rotationally symmetric.

In the standard Gierer-Meinhardt system without precursor, it is assumed that $\mu(x)\equiv1.$

The main idea underpinning
these stable spike clusters
is the following: due to the small inhibitor diffusivity
the  interaction between spikes
is repulsive and the spikes are attracted towards
the local minimum point of the precursor inhomogeneity.
Combining these two effects can lead to an equilibrium of spike positions within the cluster such that the cluster is linearly stable..
The repulsive nature of spikes has been shown in \cite{ei-wei-2002}.
The attracting feature of the local minimum of a precursor has been established in \cite{wmhgm}.

Problem (\ref{1.1}) without precursor has been studied by numerous authors. For the one-dimensional case in a bounded interval $(-1,1)$ with Neumann boundary conditions, the existence of symmetric $N$-peaked solutions (i.e. spikes of the same amplitude in leading order) was first established by I.Takagi \cite{ta1}. The existence of asymmetric $N$-spikes was first shown by Ward-Wei \cite{waw} and Doelman-Kaper-van der Ploeg \cite{dkv} independently. For symmetric $N$-peaked solutions, Iron-Ward-Wei \cite{iww} studied the stability by using matched asymptotic expansions while Ward-Wei \cite{waw} later studied the stability for asymmetric $N$-spikes. Existence and stability for symmetric spikes in one space dimension was then established rigorously by the first two authors \cite{ww6}.

For the Gierer-Meinhardt system in two dimensions, the first two authors  rigourously proved the existence and stability of multiple peaked patterns for the Gierer-Meinhardt system in the weak coupling case and the strong coupling case for symmetric spikes \cite{ww1,ww2,ww3}. Here we say that the system is in the weak coupling case if $D\rightarrow\infty$ as $\e\rightarrow0$ and in the strong coupling case if the parameter $D$ is a finite constant independent of $\e$. For more results and background on the Gierer-Meinhardt system, we refer to \cite{ww-book} and the references therein.

In fact, already in the original Gierer-Meinhardt system \cite{gm,m1,m2}, the authors have introduced precursor gradients. These precursors were proposed to model the localisation of the head structure in the coelenterate {\em Hydra}. Gradients have also been used in the Brusselator model to restrict pattern formation to some fraction of the spatial domain \cite{hk}. In this example, the gradient carries the system in and out of the pattern-forming region of the parameter range (for example across the Turing bifurcation). Thus it restricts the domain where peak formation can occur. A similar localisation effect has been used to model segmentation patterns in the fruit fly Drosophila melanogaster in \cite{h1} and \cite{lh}.

In \cite{ww4} the existence and stability of $N$-peaked steady states for the Gierer-Meinhardt system with precursor inhomogeneity has been explored.
The spikes in these patterns can vary in amplitude and have irregular spacing. In particular, the results imply that a precursor inhomogeneity can induce instability. Single-spike solutions for the Gierer-Meinhardt system with precursor including spike dynamics have been studied in \cite{wmhgm}.

Recently, the first two authors in \cite{ww5} studied the Gierer-Meinhardt system with precursor in one space dimension and proved the existence and stability of a cluster, which consists of $N$ spikes approaching the same limiting point. More precisely, they consider the existence of a steady-state spike cluster consisting of $N$ spikes near a nondegenerate local minimum point $t^0$ of the inhomogeneity $\mu$, i.e., $\mu'(t_0)=0$, $\mu''(t_0)>0$. Further, they show that this solution is linearly stable.

Now we consider this problem in two dimensions. We shall study the existence and stability of positive $k$-peaked steady-state spike clusters to (\ref{1.1}). For simplicity, we shall study the steady-state problem for positive solutions of \eqref{1.1} in the disk $B_R$ around the origin with radius $R$, which can be stated as follows:
{
\begin{equation}
\label{1.2}
\begin{cases}
\e^2\Delta A-\mu(y)A+\frac{A^2}{H}=0,\quad &\mbox{in }B_R,\\[2mm]
D\Delta H-H+A^2=0,\quad  &\mbox{in }B_R,~\\[2mm]
\frac{\partial A}{\partial \nu}=\frac{\partial H}{\partial \nu}=0, \quad &\mbox{on }\partial B_R~.
\end{cases}
\end{equation}}

Inspired by the work \cite{dkw}, where the authors constructed multi-bump ground-state solutions and the centres of these bumps are located at the vertices of a regular polygon, while each bump resembles, after a suitable scaling in the $A$-coordinate, the unique radially symmetric solution of
\begin{equation}
\label{1.3}
\Delta w-w+w^2=0~\mathrm{in}~\B,\quad 0<w(y)\rightarrow0~\mathrm{as}~|y|\rightarrow\infty,
\end{equation}
in this paper
we shall prove the existence and stability of a spike cluster located near a nondegenerate minimum point of the precursor such that the positions of the spikes form a regular polygon. We note that the presence of such patterned steady state configurations appears driven by the smallness of the relative size $\sigma^2=\e^2/D$ of the diffusion rates of the activating and inhibiting substances. However, there is some difference between our problem and the one considered in \cite{dkw}. Here, we also need to take the precursor $\mu(x)$ into consideration and further assume that the inhibitor diffusivity $D$ is very small. After introducing the transformation
$$y=\e x,~\hat{A}(x)=\frac{1}{\xi}\frac{\e^2}{D}A(\e x),
~\hat{H}(x)=\frac{1}{\xi}\frac{\e^2}{D}H(\e x),$$
and dropping hats, equation (\ref{1.2}) becomes,
\begin{align}
\label{xi}
\begin{cases}
\Delta A-\mu(\e x)A+\frac{A^2}{H}=0,\quad &\mbox{in}~B_{R/\e},\\[2mm]
\Delta H-\sigma^2H+\xi A^2=0,\quad &\mbox{in}~B_{R/\e},\\[2mm]
\frac{\partial A}{\partial \nu}=\frac{\partial H}{\partial \nu}=0, \quad &\mbox{on}~\partial B_{R/\e},
\end{cases}
\end{align}
where the explicit definition of $\xi$ will be given in (\ref{3.4}).
\medskip

Our first result is the following
\begin{theorem}
\label{th1.1}
Let $k\geq2$ be a positive integer. We assume $\mu(x)\in {C^3(B_{R})}$ be a positive, radially symmetric function and $\mu(0)=1,\mu'(0)=0,\mu''(0)>0,$ where $\mu'$ denotes the radial derivative.
Then, for
$$\max\left(\frac{\e}{\sqrt{D}},D\log\frac{\sqrt{D}}{\e}\right)\rightarrow0,$$
problem (1.2) has a $k$-spike cluster solution which concentrates at $0$. In particular, it satisfies
\begin{equation}
\label{ah}
A_{\e}(y)\sim \sum_{i=1}^k\frac{\xi D}{\e^2}w(\frac{y}{\e}-q_i),~
H_{\e}(y)\sim \frac{\xi D}{\e^2},
\end{equation}
where $\xi$ is given in \eqref{3.4} and $q_1,\cdots,q_k$ are the vertices of a $k$ regular polygon. Further, $\e q_i\rightarrow0,~i=1,\cdots,k.$
\end{theorem}

\noindent{\bf{Remark}}: Here the assumption on the value of $\mu(0)=1$ is to make the computation and representation convenient. Without loss of generality we can always do some scaling transformation for the solution $(A_{\e},H_{\e})$ to achieve the assumption $\mu(0)=1$.
\medskip

Next, we state our second result which concerns the stability of the $k$-spike cluster steady state given in Theorem \ref{th1.1}.

\begin{theorem}
\label{th1.2}
For $\max\left(\frac{\e}{\sqrt{D}},D\log\frac{\sqrt{D}}{\e}\right)$ sufficiently small, let $A_{\e},H_{\e}$ be the $k$-spike cluster steady state given in {Theorem \ref{th1.1}}. Then, for every $k$, there exists $\tau_0>0$ independent of $\e$ and $D$ such that the $k$-spike cluster steady state $(A_{\e},H_{\e})$ is linearly stable for $2\leq {k\leq3}$ and unstable for $k\geq5$ provided that $0\leq\tau<\tau_0.$
\end{theorem}

\begin{proposition}
\label{prop1.1}
For any $k\geq 2$ there also exists a $(k+1)$-spike cluster steady state similar to the one given in {Theorem \ref{th1.1}} but with an extra spike in the centre of the polygon. This $(k+1)$-spike cluster is linearly stable if $k\leq 5$ provided that $0\leq\tau<\tau_0.$
\end{proposition}

\begin{remark}
These results suggest that placing a spike in the centre of a polygon can stabilise the spike cluster in the following sense: by putting a spike in the centre of the polygon it is possible to get a stable polygonal spike cluster containing six spikes but without a centre the number of spikes for a stable cluster cannot be five or more.
\end{remark}

\begin{remark}
The stability of a spike cluster with four spikes on a regular polygon remains open. The stability problem in this case requires further expansion of an eigenvalue which is zero in the two leading orders.
\end{remark}

\begin{remark}
The stability of a spike cluster with spikes on a regular polygon with six or more vertices plus its centre remains open. The stability problem in this case requires some new analysis since the interaction between spikes is of a different type from the one considered in this paper.
\end{remark}







\noindent

\medskip

This paper is organised as follows. In section 2, we present some preliminaries on the properties of the nonlocal linear operators which will appear in the study of existence and the analysis of large eigenvalues of order $O(1)$ in the stability proof, also we will give the basic asymptotic behaviour of the Green function of $\Delta-I$ in $B_R$. We study the existence of a $k$-spike cluster solution to (\ref{1.2}) in section 3 (Liapunov-Schmidt reduction) and section 4 (solving the reduced problem). In appendix A we prove some auxiliary results which are needed in section 3. In section 5, we rigourously study the large eigenvalues of order $O(1)$ {for} the linearised problem around the steady state spike cluster. The small eigenvalues of order $o(1)$ will be investigated in appendix B (general theory) and section 6 (explicit computation of small eigenvalues which decide the stability of spike clusters). In section 7 we sketch how the approach can be adapted to show existence and stability of a cluster for which the spikes are located at the vertices of a regular polygon with centre.

\medskip

Throughout the paper, by $C,c$ we denote generic constants which may change from line to line. Further, $h.o.t.$ stands for higher order terms.

\vspace{1cm}
\section{Preliminaries: Scaling property, Green's function and Eigenvalue problem}
In this section we shall provide some preliminaries which will be needed for the existence and stability proofs.

Let $w$ be the ground state solution given in (\ref{1.3}), i.e, the unique solution of the problem
\begin{equation}
\label{2.1}
\begin{cases}
&\Delta w-w+w^2=0,~y\in\mathbb{R}^2,\quad w>0,\\
&w(0)=\max_{y\in\B}w(y),~w(y)\rightarrow0~\mathrm{as}~|y|\rightarrow0.
\end{cases}
\end{equation}
Let
\begin{equation}
\label{2.2}
L_0\phi=\Delta\phi-\phi+2w\phi,~\phi\in H^2(\B).
\end{equation}
We first recall the following well known result:

\begin{lemma}
\label{le2.1}
The eigenvalue problem
\begin{equation}
\label{2.3}
L_0\phi=\lambda\phi,~\phi\in H^2(\B),
\end{equation}
admits the following set of eigenvalues
\begin{equation}
\label{2.4}
\lambda_1>0,~\lambda_2=\lambda_3=0,~\lambda_4<0,{\cdots.}
\end{equation}
The eigenfunction $\Phi_0$ corresponding to $\lambda_1$ can be made positive and radially symmetric, the space of eigenfunctions corresponding to the eigenvalue $0$ is
$$K_0:=span\left\{\frac{\partial w}{\partial y_j},j=1,2\right\}.$$
\end{lemma}

\begin{proof}
This lemma follows from \cite[Theorem 2.1]{ln} and {\cite[Lemma C]{nt2}.}
\end{proof}

Next, we consider the following nonlocal eigenvalue problem:
\begin{equation}
\label{2.5}
\Delta\phi-\phi+2w\phi-2\frac{\int_{\B}w\phi}{\int_{\B}w^2}w^2=\alpha_0\phi,~\phi\in H^2(\B).
\end{equation}
Problem (\ref{2.5}) plays a key role in the study of large eigenvalues (See section 5 below). For problem \eqref{2.5}, we have the following theorem due to \cite[Theorem 1.4]{w2}.

\begin{theorem}
\label{th2.1}
Let $\alpha_0\neq0$ be an eigenvalue of the problem (\ref{2.5}). Then we have $\Re(\alpha_0)\leq-c_1$ for some $c_1>0.$
\end{theorem}

We shall also consider the following system of nonlocal eigenvalue problems:
\begin{equation}
\label{2.6}
L\Phi:=\Delta\Phi-\Phi+2w\Phi-2\frac{\int_{\B}w\Phi}{\int_{\B}w^2}w^2,
\end{equation}
where
\begin{align*}
\Phi:=\left(\phi_1,\phi_2,\cdots,\phi_k\right)^T\in(H^2(\B)^k).
\end{align*}
Set
\begin{equation}
\label{2.7}
L_0\phi:=\Delta\phi-\phi+2w\phi,
\end{equation}
where $\phi\in H^2(\B)$. Then the conjugate operator of $L$ under the scalar product in $L^2(\B)$ is given by
\begin{equation}
\label{2.8}
L^*\Phi=\Delta\Phi-\Phi+2w\Phi-2\frac{\int_{\B}w^2\Phi}{\int_{\B}w^2}w.
\end{equation}
We have the following result
\begin{lemma}
\label{le2.2}
We have
\begin{equation*}
Ker(L)=K_0\oplus K_0\oplus\cdots\oplus K_0,
\end{equation*}
and
\begin{equation*}
Ker(L^*)=K_0\oplus K_0\oplus\cdots\oplus K_0.
\end{equation*}
\end{lemma}

\begin{proof}
The system (\ref{2.6}) is in diagonal form. Suppose
\begin{equation*}
L\Phi=0.
\end{equation*}
For $i=1,2,\cdots,k$ the $i$-th equation of (\ref{2.6}) is given by
\begin{equation}
\label{2.9}
\Delta\phi_i-\phi_i+2w\phi_i-2\frac{\int_{\B}w\phi_i}{\int_{\B}w^2}w^2=0.
\end{equation}
We claim that (\ref{2.9}) admits only the solution $\frac{\partial w}{\partial x_i},i=1,2$. Indeed, we note that
$\phi_i'=\frac{\int_{\B}w\phi_i}{\int_{\B}w^2}w$ satisfies that
\begin{equation*}
\Delta\phi_i'-\phi_i'+2w\phi_i'=\frac{\int_{\B}w\phi_i}{\int_{\B}w^2}w^2.
\end{equation*}
As a result, $\phi_i-{2}\phi_i'$ satisfies
\begin{equation*}
\Delta(\phi_i-2\phi_i')-(\phi_i-2\phi_i')+2w(\phi_i-2\phi_i')=0,
\end{equation*}
and we get
\begin{equation}
\label{2.10}
\phi_i-2\phi_i'=c_1\frac{\partial w}{\partial x_1}+c_2\frac{\partial w}{\partial x_2}
\end{equation}
by Lemma \ref{le2.1}. Multiplying by $w$ on both sides of (\ref{2.10}) and integrating, we have $\int_{\B}w\phi_i=0$. Hence, $\phi_i$ is the solution to
\begin{equation*}
\Delta\phi-\phi+2w\phi=0,
\end{equation*}
and we get the first conclusion of Lemma \ref{le2.2}. To prove the second statement, we proceed in a similar way for $L^*$, and the $i$-th equation of (\ref{2.8}) is given as
\begin{equation}
\label{2.11}
\Delta\phi_i-\phi_i+2w\phi_i-2\frac{\int_{\B}w^2\phi_i}{\int_{\B}w^2}w=0.
\end{equation}
Multiplying (\ref{2.11}) by $w$ and integrating, we obtain $\int_{\B}w^2\phi_i=0.$ Then we have
\begin{equation*}
\Delta\phi_i-\phi_i+2w\phi_i=0.
\end{equation*}
By Lemma \ref{le2.1} again, we get the second conclusion and the proof is finished.
\end{proof}

By the result of Lemma \ref{le2.2}, we have
\begin{lemma}
\label{le2.3}
The operator
\begin{align*}
L:(H^2(\B))^k\rightarrow{(L^2(\B))^k}, L\Phi=\Delta\Phi-\Phi+2w\Phi-2\frac{\int_{\B}w\Phi}{\int_{\B}w^2}w^2,
\end{align*}
is invertible if it is restricted as follows:
\begin{align*}
L:\left(K_0\oplus\cdots\oplus K_0\right)^{\perp}\cap(H^2(\B))^k\rightarrow\left(K_0\oplus\cdots\oplus K_0\right)^{\perp}\cap(L^2(\B))^k.
\end{align*}
Moreover, $L^{-1}$ is bounded.
\end{lemma}

\begin{proof}
This results follows from the Fredholm Alternative and Lemma \ref{le2.2}.
\end{proof}

Next, we study the eigenvalue problem for $L:$
\begin{align}
\label{2.12}
L\Phi=\alpha\Phi.
\end{align}
We have

\begin{lemma}
\label{le2.4}
For any nonzero eigenvalue $\alpha$ of (\ref{2.12}) we have $\Re(\alpha)\leq-c<0$.
\end{lemma}

\begin{proof}
Let $(\Phi,\alpha)$ satisfy the system (\ref{2.12}). Suppose $\Re(\alpha)\geq0$ and $\alpha\neq0$. The $i$-th equation of (\ref{2.12}) becomes
\begin{equation*}
\Delta\phi_i-\phi_i+2w\phi_i-2\frac{\int_{\B}w\phi_i}{\int_{\B}w^2}\phi_i=\alpha\phi_i.
\end{equation*}
By Theorem \ref{th2.1}, we conclude that
$$\Re(\alpha)\leq-c<0.$$
\end{proof}

To conclude this section, we give a short summary on the Green function of $\Delta-1$ in $B_R.$ Let $G(x,y)$ be the Green function given by
\begin{equation*}
\begin{cases}
\Delta_x G(x,z)-G(x,z)+\delta_z(x)=0~&\mathrm{in}~B_R,\\
\frac{\partial G(x,z)}{\partial\nu}=0~&\mathrm{on}~\partial B_R.
\end{cases}
\end{equation*}
Then the following holds:
{\begin{enumerate}
  \item [(1)] If $0<|x-z|\ll1$, we have
  \begin{equation}
  \label{2.13}
  G(x,z)=\frac{1}{2\pi}\log\frac{1}{|x-z|}+H(x,z),
  \end{equation}
  where $H(x,z)$ is a continuous function and $\partial_x H(x,z)\mid_{x=z}=0.$
  \item [(2)] If $1\ll|x-z|\ll R$, we have
  \begin{equation}
  \label{2.14}
  G(x,z)=C|x-z|^{-\frac12}e^{-|x-z|}\left(1+o(1)\right),~G'(x,z)=-G(x,z)(1+o(1))
  \end{equation}
  for some generic constant $C$.
\end{enumerate}}
\vspace{1cm}

\section{Existence I: Reduction to Finite Dimensions}
In this section, we shall reduce the existence problem to a finite dimensional problem. In the first step, we choose a good approximation to an equilibrium state. Then we shall use the Liapunov-Schmidt reduction to reduce the original problem to a finite dimensional one.
Some technical proofs are done in appendix A. In the next section, we solve the reduced problem.

First of all, let us set the candidate points for the location of spikes to the activator. Let $Q_{\e}$ denote the set of the vertex points of the regular $k$-polygon.
\begin{align}
\label{3.1}
Q_{\e}=\left\{{\bf q}=(q_1,\cdots,q_k)\mid q_i=\left(2R_{\e}\cos\frac{2(i-1)\pi}{k},2R_{\e}\sin\frac{2(i-1)\pi}{k}\right)\right\},
\end{align}
where $R_{\e}$ is chosen such that
\begin{align}
\label{3.2}
\frac{1}{C}\frac{\sqrt{D}}{\e}\log\left(\frac{1}{D\log\frac{\sqrt{D}}{\e}}\right)\leq R_{\e}\leq C\frac{\sqrt{D}}{\e}\log\left(\frac{1}{D\log\frac{\sqrt{D}}{\e}}\right)
\end{align}
for some constant $C$ independent of $\e$ and $D$. If $\max\left(\frac{\e}{\sqrt{D}},D\log\frac{\sqrt{D}}{\e}\right)\rightarrow0$, then we can see that $\e R_{\e}\rightarrow0$ and $\frac{\sqrt{D}}{\e R_{\e}}\rightarrow0.$
\medskip

Recall that we want to solve (\ref{xi}) which is given by
\begin{align*}
{\begin{cases}
\Delta A-\mu(\e x)A+\frac{A^2}{H}=0,\quad&\mathrm{in}~B_{R/\varepsilon},\\[2mm]
\Delta H-\sigma^2H+\xi A^2=0,\quad&\mathrm{in}~B_{R/\varepsilon},\\[2mm]
\frac{\partial A}{\partial \nu}=\frac{\partial H}{\partial \nu}=0, \quad&\mbox{on}~\partial B_{R/\e}.
\end{cases}}
\end{align*}
where $\sigma^2=\frac{\e^2}{D}$.

Next we introduce the cut off function $\chi_{\e,q_j}(x)=\chi(\frac{x-q_j}{R_{\e}\sin\frac{\pi}{k}})$, where
\begin{equation}
\label{5.test}
\chi(x)=\begin{cases}1,~|x|\leq\frac12,\\[2mm]0,~|x|>1,\end{cases}~\chi\in C_0^{\infty}(\B).
\end{equation}

 Let $(q_1,\cdots,q_k)$ be defined as in (\ref{3.1}) and we set
\begin{align*}
W=\sum_{j=1}^k\chi_{\varepsilon,q_j}(x)w(x-q_j).
\end{align*}

We write
\begin{align}
\label{3.4}
\xi^{-1}=\int_{B_{R/\varepsilon}}G(\sigma q_1,\sigma z)\left(\sum_{j=1}^k\chi(\varepsilon z)w(z-q_j)\right)^2\mathrm{d}z=
\frac{1}{2\pi}\log\frac{1}{\s}({\int_{\mathbb{R}^2}w^2}+o(1)).
\end{align}
\medskip

For a function $u\in H^2(B_{R/\varepsilon})$, let $T[u]$ be the unique solution
 in $H^2_N(B_{R/\varepsilon})$
 of the following problem:
\begin{align}
\label{3.5}
\Delta T[u]-\sigma^2 T[u]+\xi u^2=0~\mathrm{in}~B_{R/\varepsilon},
\end{align}
where
\[
H^2_N(B_{R/\varepsilon})=\{u\in H^2(\BB)\,|\, \frac{\partial u}{\partial \nu}=0 \ \mbox { on }
\partial \BB\}.
\]
Written differently, we have
\begin{align}
\label{3.6}
T[u](x)=\xi\int_{B_{R/\varepsilon}}G_{\sigma}(x,z)u^2(z)\mathrm{d}y,
\end{align}
where $G_{\sigma}(x,z)$ is the Green function {which satisfies $(\Delta-\sigma^2)G_{\sigma}(x,z)+\delta_z(x)=0$ in $\BB$} with Neumann boundary condition.

\medskip

System (\ref{xi}) is equivalent to the following equation in operator form:
\begin{align}
\label{3.7}
S_{\e}(u,v)=
\left(\begin{matrix}
S_1(A,H)\\S_2(A,H)\end{matrix}\right)=0,~
H^2_N(B_{R/\varepsilon})\times H^2_N(B_{R/\varepsilon})\rightarrow L^2(B_{R/\varepsilon})\times L^2(B_{R/\varepsilon}),
\end{align}
where
\begin{align*}
S_1(A,H)=\Delta A-\mu(\e x)A+\frac{A^2}{H}:~H^2_N(B_{R/\varepsilon})\times H^2_N(B_{R/\varepsilon})\rightarrow L^2(B_{R/\varepsilon}),\\
S_2(A,H)=\Delta H-\s^2H+\xi A^2:~H^2_N(B_{R/\varepsilon})\times H^2_N(B_{R/\varepsilon})\rightarrow L^2(B_{R/\varepsilon}).
\end{align*}
For equation (\ref{xi}), we choose our approximate solution as follows,
\begin{align}
\label{3.8}
A_{\e,\q}=W,~H_{\e,\q}=T[W].
\end{align}
Note that $H_{\e,\q}$ satisfies
\begin{align*}
0=\Delta H_{\e,\q}-\s^2H_{\e,\q}+\xi A_{\e,\q}^2=\Delta H_{\e,\q}-\s^2H_{\e,\q}+\xi\sum_{j=1}^kw(x-q_j)^2+h.o.t..
\end{align*}
Further, by our choice of $\xi$ in \eqref{3.4}, it is easy to see that $H_{\e,\q}(q_i)=1,~i=1,\cdots,k.$
We insert our ansatz (\ref{3.8}) into (\ref{xi}) and calculate
\begin{equation}
\label{3.9}
S_2(A_{\e,\q},H_{\e,\q})=0,
\end{equation}
and
\begin{align}
\label{3.10}
S_1(A_{\e,\q},H_{\e,\q})=&~\Delta A_{\e,\q}-\mu(\e x)A_{\e,\q}+\frac{A_{\e,\q}^2}{H_{\e,\q}}\nonumber\\
=&\sum_{j=1}^k\Big[\Delta w(x-q_j)-\mu(\e x)w(x-q_j)\Big]+\sum_{j=1}^kw^2(x-q_j)H_{\e,\q}^{-1}+h.o.t.\nonumber\\
=&\sum_{j=1}^k\big(1-\mu(\e x)\big)w(x-q_j)+\sum_{j=1}^kw^2(x-q_j)\big(H_{\e,\q}^{-1}-1\big)+h.o.t.
\end{align}
On the other hand, we calculate for $j=1,\cdots,k$ and $x=q_j+ z$ with $|\s z|<\delta:$
\begin{align}
\label{3.11}
H_{\e,\q}(q_j+ z)-1=~&\xi\int_{B_{R/\varepsilon}}\Big(G_{\s}(q_j+z,t)-G_{\s}(q_j,t)\Big)A^2_{\e,\q}\mathrm{d}t\nonumber\\
=~&\xi\int_{B_{R/\varepsilon}}\Big(G_{\s}(q_j+z,t)-G_{\s}(q_j,t)\Big)w(t-q_j)^2\mathrm{d}t\nonumber\\
&+\xi\int_{B_{R/\varepsilon}}\Big(G_{\s}(q_j+z,t)-G_{\s}(q_j,t)\Big)\sum_{l\neq j}w(t-q_l)^2\mathrm{d}t+O\left(e^{-2R_{\e}\sin(\frac{\pi}{k})}\right)\nonumber\\
=~&\xi\int_{\B}\frac{1}{2\pi}\log\frac{|t|}{|z-t|}w^2(t)\mathrm{d}t+\xi\left(\sum_{l=1}^2\frac{\partial F(\q)}{\partial q_{j,l}}\sigma z_l\int_{\B}w^2(t)\mathrm{d}t\right)\nonumber\\
&+\xi\sum_{l,m}\frac{\partial^2F(\q)}{\partial q_{j,l}\partial q_{j,m}}\sigma^2z_lz_m\int_{\B}w^2+O\left(e^{-2R_{\e}\sin(\frac{\pi}{k})}\right)\nonumber\\
&+O\left(\sigma^3|z|^2+\s^2R_{\s}^{-\frac12}e^{-R_{\s}}|z|\right),
\end{align}
where $R_{\s}=2\sigma R_{\e}\sin\left(\frac{\pi}{k}\right)$ and
\begin{align}
\label{fq}
F(\q)=\sum_{i=1}^kH(\s q_i,\s q_i)+\sum_{i\neq j}^kG(\s q_i,\s q_j).
\end{align}

Substituting (\ref{3.11}) into (\ref{3.10}), we have the following key estimate,
\begin{lemma}
\label{le3.1}
For $x=q_j+z,~|\s z|<\delta$, we have
\begin{align}
\label{3.12}
S_1(A_{\e,\q},H_{\e,\q})=S_{1,1}+S_{1,2},
\end{align}
where
\begin{align}
\label{3.13}
S_{1,1}(z)=~&\xi H_{\e,\q}(q_j)^{-2}\left(\int_{\B}w^2\right)w^2(z)\Big(\sigma z\nabla_{q}F(\q)
+\sigma^2z_lz_m\frac{\partial^2F(\bf q)}{\partial q_{j,l}\partial q_{j,m}}+h.o.t.\Big)\nonumber\\
&+\e\sum_{l=1}^2z_l\left(\mu''(0)+\frac12\mu'''(0)\e|q_j|+O(\e^2|q_j|^2)\right)\e q_{j,l}\sum_{i=1}^kw(x-q_i)\nonumber\\
&+O\left(\s^3|z|^2+\s^2R_{\s}^{-\frac12}e^{-R_{\s}}|z|+e^{-2R_{\e}\sin\left(\frac{\pi}{k}\right)}\right),
\end{align}
and
\begin{align}
\label{3.14}
S_{1,2}(z)=\xi w^2(z)R(|z|)+\e^2R_{\e}^2w(z),
\end{align}
where $R(|z|)$ is a radially symmetric function with the property that $$R(|z|)=O(\log(1+|z|)).$$
Further, $S_1(A_{\e,\q},H_{\e,\q})=e^{-\frac{\delta}{\s}}$ for $|x-q_j|\geq\frac{\delta}{\sigma},~j=1,2,\cdots,k.$
\end{lemma}

The above estimates will be very important in the following calculations, where (\ref{3.7}) is solved exactly.

Now we study the linearised operator defined by
\begin{align*}
\tilde{L}_{\e,\q}:=S_{\e,\q}'\left(\begin{matrix}A_{\e,\q}\\H_{\e,\q}\end{matrix}\right),
\end{align*}
\begin{align*}
\tilde{L}_{\e,\q}:H^2_N(B_{R/\varepsilon})\times H^2_N(B_{R/\varepsilon})\rightarrow L^2(B_{R/\varepsilon})\times L^2(B_{R/\varepsilon}).
\end{align*}

Set
\begin{align*}
K_{\e,\q}:=\mathrm{span}\left\{\frac{\partial A_{\e,\q}}{\partial q_{j,l}}\mid j=1,2,\cdots,k,l=1,2\right\}\subset H^2_N(B_{R/\varepsilon}),
\end{align*}
and
\begin{align*}
C_{\e,\q}:=\mathrm{span}\left\{\frac{\partial A_{\e,\q}}{\partial q_{j,l}}\mid j=1,2,\cdots,k,l=1,2\right\}\subset L^2(B_{R/\varepsilon}).
\end{align*}

The operator $\tilde{L}_{\e,\q}$ is not uniformly invertible in $\e$ and $\sigma$ due to approximate kernel,
\begin{align}
\label{3.15}
\mathcal{K}_{\e,\q}:=K_{\e,\q}\oplus\{0\}\subset H^2_N(B_{R/\varepsilon})\times H^2_N(B_{R/\varepsilon}).
\end{align}
We choose the approximate cokernel as follows:
\begin{align}
\label{3.16}
\mathcal{C}_{\e,\q}:=C_{\e,\q}\oplus\{0\}\subset L^2(B_{R/\varepsilon})\times L^2(B_{R/\varepsilon}).
\end{align}
{Then we} define
\begin{align}
\label{3.17}
\mathcal{K}_{\e,\q}^{\perp}:=K_{\e,\q}^{\perp}\oplus H^2_N(B_{R/\varepsilon})\subset H^2_N(B_{R/\varepsilon})\times H^2_N(B_{R/\varepsilon}),
\end{align}
\begin{align}
\label{3.18}
\mathcal{C}_{\e,\q}^{\perp}:=C_{\e,\q}^{\perp}\oplus L^2(B_{R/\varepsilon})\subset L^2(B_{R/\varepsilon})\times L^2(B_{R/\varepsilon}),
\end{align}
where $\mathcal{C}_{\e,\q}^{\perp}$ and $\mathcal{K}_{\e,\q}^{\perp}$ denote the orthogonal complement with the scalar product of $L^2(B_{R/\varepsilon})$ in the spaces $H^2_N(B_{R/\varepsilon})$ and $L^2(B_{R/\varepsilon})$, respectively.

Let $\pi_{\e,\q}$ denote the projection in $L^2(B_{R/\varepsilon})\times L^2(B_{R/\varepsilon})$ onto $\mathcal{C}_{\e,\q}^{\bot}$, where the second component of the projection is the identity map. We are going to show that the equation
\begin{align*}
\pi_{\e,\q}\circ S_{\e,\q}\left(\begin{matrix}
A_{\e,\q}+\phi_{\e,\q}\\
H_{\e,\q}+\psi_{\e,\q}
\end{matrix}\right)=0
\end{align*}
has a unique solution $\Sigma_{\e,\q}=\left(\begin{matrix}
\phi_{\e,\q}\\
\psi_{\e,\q}
\end{matrix}\right)\in\mathcal{K}_{\e,\q}^{\perp}$ if $\max\left(\frac{\e}{\sqrt{D}},D\log\frac{\sqrt{D}}{\e}\right)$ is small enough.

Set
\begin{align}
\label{3.19}
\mathcal{L}_{\e,\q}=\pi_{\e,\q}\circ\tilde{L}_{\e,\q}:\mathcal{K}_{\e,\q}^{\perp}\rightarrow\mathcal{C}_{\e,\q}^{\perp}.
\end{align}

In appendix A we will show that this linear operator is uniformly invertible. Then we will use this to prove the existence of $\Sigma_{\e,\q}=\left(\begin{matrix}
\phi_{\e,\q}\\
\psi_{\e,\q}
\end{matrix}
\right).$

\vspace{1cm}
\section{Existence II: The reduced problem}
In this section, we solve the reduced problem and prove Theorem \ref{th1.1}.

By Lemma \ref{le3.2}, for each $\q\in Q_{\e}$, there exists a unique solution $(\Phi_{\e,\q},\Psi_{\e,\q})\in\mathcal{K}_{\e,\q}^{\perp}$ such that
\begin{equation*}
S_{\e,\q}\left(\begin{matrix}A_{\e,\q}+\Phi_{\e,\q}\\H_{\e,\q}+\Psi_{\e,\p}\end{matrix}\right)
=\left(\begin{matrix}\Xi_{\e,\q}\\0\end{matrix}\right)\in\mathcal{C}_{\e,\q}.
\end{equation*}
Our idea is to find $\q$ such that $S_{\e,\q}\left(\begin{matrix}A_{\e,\q}+\Phi_{\e,\q}\\H_{\e,\q}+\Psi_{\e,\q}\end{matrix}\right)\perp \mathcal{C}_{\e,\q}.$
Let
\begin{align*}
W_{\e,j,i}(\q):=\frac{1}{\xi}\int_{B_{R/\varepsilon}}\left(S_{1}\big(A_{\e,\q}+\Phi_{\e,\q},H_{\e,\q}+\Psi_{\e,\q}\big)\frac{\partial A_{\e,\q}}{\partial q_{j,i}}\right),
\end{align*}
where $j=1,2,\cdots,k$ and $i=1,2.$ We set
\begin{align*}
W_{\e}(\q)=\left(W_{\e,1,1}(\q),\cdots,W_{\e,k,2}(\q)\right).
\end{align*}
It is easy to see that $W_{\e}(\q)$ is a map which is continuous in $\q,$ and our problem is reduced to finding a zero of the vector field $W_{\e}(\q)$. Since the points $q_1,q_2,\cdots,q_k$ are the vertices of a regular $k$-polygon and $\mu(x)$ is a radially symmetric function, if we can find $\q\in Q_{\e}$ such that $\left(W_{\e,1,1}(\q),W_{\e,1,2}(\q)\right)=0$, then $W_{\e}(\q)=0$. Further, we note that the approximate solution $(A_{\e,\q},H_{\e,\q})$ is invariant under rotation by $\frac{2\pi}{k}.$ Thus, using \cite[Corollary 7.1]{dkw}, $W_{\e,1,2}$ equals $0$. So, all that remains is finding $\q$ such that $W_{\e,1,1}(\q)=0.$
\medskip

We calculate the asymptotic expansion of $W_{\e,1,1}(\q)$,
\begin{align*}
&\int_{B_{R/\varepsilon}}S_1(A_{\e,\q}+\Phi_{\e,\q},H_{\e,\q}+\Psi_{\e,\q})\frac{\partial A_{\e,\q}}{\partial q_{1,1}}\nonumber\\
&\quad=\int_{B_{R/\varepsilon}}\left[\Delta(A_{\e,\q}+\Phi_{\e,\q})-\mu(A_{\e,\q}+\Phi_{\e,\q})
+\frac{(A_{\e,\q}+\Phi_{\e,\q})^2}{H_{\e,\q}+\Psi_{\e,\q}}\right]\frac{\partial A_{\e,\q}}{\partial q_{1,1}}\nonumber\\
&\quad=\int_{B_{R/\varepsilon}}\left[\Delta(A_{\e,\q}+\Phi_{\e,\q})-(A_{\e,\q}+\Phi_{\e,\q})
+\frac{(A_{\e,\q}+\Phi_{\e,\q})^2}{H_{\e,\q}}\right]\frac{\partial A_{\e,\q}}{\partial q_{1,1}}\nonumber\\
&\quad\quad+\int_{B_{R/\varepsilon}}\left[\frac{(A_{\e,\q}+\Phi_{\e,\q})^2}{H_{\e,\q}+\Psi_{\e,\q}}-
\frac{(A_{\e,\q}+\Phi_{\e,\q})^2}{H_{\e,\q}}\right]\frac{\partial A_{\e,\q}}{\partial q_{1,1}}\nonumber\\
&\quad\quad+\int_{B_{R/\varepsilon}}\left[(1-\mu)(A_{\e,\q}+\Phi_{\e,\q})\right]\frac{\partial A_{\e,\q}}{\partial q_{1,1}}\nonumber\\
&\quad=I_1+I_2+I_3,
\end{align*}
where $I_i,i=1,2,3$ are defined at the last equality.

For $I_1$, we have by Lemma \ref{le3.2},
\begin{align}
\label{4.1}
I_1=&\int_{B_{R/\varepsilon}}\left[\Delta(A_{\e,\q}+\Phi_{\e,\q})-(A_{\e,\q}+\Phi_{\e,\q})
+\frac{(A_{\e,\q}+\Phi_{\e,\q})^2}{H_{\e,\q}(q_1)}\right]\frac{\partial A_{\e,\q}}{\partial q_{1,1}}\nonumber\\
&-\int_{B_{R/\varepsilon}}\frac{(A_{\e,\q}+\Phi_{\e,\q})^2}{H_{\e,\q}^2(q_1)}(H_{\e,\q}-H_{\e,\q}(q_1))\frac{{\partial A_{\e,\q}}}{\partial q_{1,1}}+O(e^{-2R_{\e}\sin\frac{\pi}{k}})\nonumber\\
=&-\int_{B_{R/\varepsilon}}\left[\Delta(w_1+\Phi_{\e,\q})-(w_1+\Phi_{\e,\q})
+\frac{(w_1+\Phi_{\e,\q})^2}{H_{\e,\q}(q_1)}\right]
\frac{\partial w_1}{\partial x_{1}}\nonumber\\
&+\int_{B_{R/\varepsilon}}\frac{(w_1+\Phi_{\e,\q})^2}{H^2_{\e}(q_1)^2}(H_{\e,\q}(q_1+z)-H_{\e,\q}(q_1))\frac{\partial w_1}{\partial x_1}+O(e^{-2R_{\e}\sin\frac{\pi}{k}}),
\end{align}
where $w_1=w(x-q_1).$ Note that by Lemma \ref{le3.2}, we have $\Phi_{\e,\q,2}$ is radially symmetric with respect to $z$. Then we have
\begin{align}
\label{4.2}
\int_{B_{R/\varepsilon}}\Big[\Delta \Phi_{\e,\q}-\Phi_{\e,\q}+2w_{1}\Phi_{\e,\q}\Big]\frac{\partial w_{1}}{\partial x_1}&=\int_{B_{R/\varepsilon}}\Phi_{\e,\q,1}\frac{\partial}{\partial x_1}[\Delta w-w+w^2]=0,
\end{align}
and
\begin{align}
\label{4.3}
\int_{B_{R/\varepsilon}}(\Phi_{\e,\q})^2\frac{\partial w_1}{\partial x_1}=~&\int_{B_{R/\varepsilon}}\Phi_{\e,\q,1}\Phi_{\e,\q,2}\frac{\partial w_1}{\partial x_1}\nonumber\\
=~&{O\left(\s(\log\frac{1}{\s})^{-2}R_{\s}^{-\frac12}e^{-R_{\s}}+
\s^2(\log\frac{1}{\s})^{-2}+(\log\frac{1}{\s})^{-1}\e^2R_{\e}\right).}
\end{align}

From (\ref{4.1})-(\ref{4.3}), we get
\begin{align}
\label{4.4}
I_1=&\int_{B_{R/\varepsilon}}w_1^2(H_{\e,\q}(q_1+z)-H_{\e,\q}(q_1))\frac{\partial w_1}{\partial x_1}+h.o.t.\nonumber\\
=&\xi\s\sum_{k=1}^2\frac{\partial F(\q)}{\partial q_{1,k}}\int_{\B}w^2z_k\frac{\partial w}{\partial z_1}\int_{\B}w^2+h.o.t.\nonumber\\
=&-c_1\xi\s\frac{\partial F(\q)}{\partial q_{1,1}}+h.o.t.,
\end{align}
where $F(\q)$ is defined in (\ref{fq}),  $c_1=\frac13\int_{\B}w^2\int_{\B}w^3$ and $h.o.t.$ represent terms of the order
$$\s(\log\frac{1}{\s})^{-2}R_{\s}^{-\frac12}e^{-R_{\s}}+\s^2(\log\frac{1}{\s})^{-2}
+(\log\frac{1}{\s})^{-1}\e^2R_{\e}.$$
\medskip

Next, we study the term $I_2.$ We recall that $\Psi_{\e,\q}$ satisfies the following equation
\begin{equation}
\label{4.5}
\Delta\Psi_{\e,\q}-\sigma^2\Psi_{\e,\q}+2\xi A_{\e,\q}\Phi_{\e,\q}+\xi\Phi_{\e,\q}^2=0.
\end{equation}
As for the perturbation term $\Phi_{\e,\q}$, we can also make a decomposition for $\Psi_{\e,\q}=\Psi_{\e,\q,1}+\Psi_{\e,\q,2},$ where
\begin{align*}
\Delta\Psi_{\e,\q,1}-\sigma^2\Psi_{\e,\q,1}+2\xi A_{\e,\q}\Phi_{\e,\q,1}+\xi(\Phi_{\e,\q,1}^2+2\Phi_{\e,\q,1}\Phi_{\e,\q,2})=0,
\end{align*}
and
\begin{align*}
\Delta\Psi_{\e,\q,2}-\sigma^2\Psi_{\e,\q,2}+2\xi A_{\e,\q}\Phi_{\e,\q,2}+\xi\Phi_{\e,\q,2}^2=0.
\end{align*}
Then we can easily see that
$$\|\Psi_{\e,\q,1}\|_{H^2(B_{R/\varepsilon})}=O\left(\s(\log\frac{1}{\s})^{-1}R_{\s}^{-\frac12}e^{-R_{\s}}+
\s^2(\log\frac{1}{\s})^{-1}+\e^2 R_{\e}\right)$$
and $\Psi_{\e,\q,2}$ is radially symmetric with respect to $z$. Further, from the Green representation formula we get that
{\begin{align}
\label{4.6}
\Psi_{\e,\q,1}(q_1+z)-\Psi_{\e,\q}(q_1)
=~&\xi\int_{B_{R/\varepsilon}}(G_{\sigma}(p_1,q_1+z)-G_\s(p_1,z))(2A_{\e,\q}\Phi_{\e,\q}+\Phi_{\e,\q}^2)\,dz
\nonumber\\
=~&o(1)\xi\s|\nabla q_1F(\q)||z|+R_1(|z|),
\end{align}}
where $R_1(|z|)$ is a radially symmetric function.

Substituting (\ref{4.5}) and \eqref{4.6} into $I_2,$ we get
{\begin{align}
\label{4.7}
I_2=&\int_{B_{R/\varepsilon}}\left[\frac{(A_{\e,\q}+\Phi_{\e,\q})^2}{H_{\e,\q}+\Psi_{\e,\q}}
-\frac{(A_{\e}+\Phi_{\e,\q})^2}{H_{\e,\q}}\right]\frac{\partial A_{\e,\q}}{\partial q_{1,1}}\nonumber\\
=&-\int_{B_{R/\varepsilon}}\frac{(A_{\e,\q}+\Phi_{\e,\q})^2}{H_{\e,\q}^2}\Psi_{\e,\q}\frac{\partial A_{\e,\q}}{\partial q_{1,1}}
+O\left(\s(\log\frac{1}{\s})^{-2}R_{\s}^{-\frac12}e^{-R_{\s}}+\s^2(\log\frac{1}{\s})^{-2}
+(\log\frac{1}{\s})^{-1}\e^2R_{\e}\right)\nonumber\\
=&-\int_{B_{R/\varepsilon}}\frac13\frac{\partial w_{1}^3}{\partial y_1}(\Psi-\Psi(q_1))
+O\left(\s(\log\frac{1}{\s})^{-2}R_{\s}^{-\frac12}e^{-R_{\s}}+\s^2(\log\frac{1}{\s})^{-2}
+(\log\frac{1}{\s})^{-1}\e^2R_{\e}\right)\nonumber\\
=&~o(1)\xi\s|\nabla q_1F(\q)|+O\left(\s(\log\frac{1}{\s})^{-2}R_{\s}^{-\frac12}e^{-R_{\s}}+\s^2(\log\frac{1}{\s})^{-2}
+(\log\frac{1}{\s})^{-1}\e^2R_{\e}\right).
\end{align}}

For $I_3,$ we have
\begin{align}
\label{4.8}
I_3=&\int_{B_{R/\varepsilon}}(1-\mu)w(x-q_1)\frac{\partial w(x-q_1)}{\partial q_{1,1}}+O(e^{-2R_{\e}\sin\left(\frac{\pi}{k}\right)})\nonumber\\
=&\int_{B_{R/\varepsilon}}(\mu-\mu(\e q_1))w_1\frac{\partial w_1}{\partial q_{1,1}}+\int_{B_{R/\varepsilon}}(\mu(\e q_1)-\mu(\e x))w_1\frac{\partial w_1}{\partial q_{1,1}}
+O(e^{-2R_{\e}\sin\left(\frac{\pi}{k}\right)})\nonumber\\
=&\int_{B_{R/\varepsilon}}\Big[\partial_1\mu(\e q_1)\e z_1+\partial_2\mu(\e q_1)\e z_2\Big]w(z)\frac{\partial w(z)}{\partial z_1}
+O(e^{-2R_{\e}\sin\left(\frac{\pi}{k}\right)})\nonumber\\
=&~\e \partial_1\mu(\e q_1)\int_{\B}w\frac{\partial w}{\partial r}\frac{x^2}{r}
+O(e^{-2R_{\e}\sin\left(\frac{\pi}{k}\right)})\nonumber\\
=&~c_2\e^2R_{\e}\partial^2\mu(0)+O(e^{-2R_{\e}\sin\left(\frac{\pi}{k}\right)}+\e^3 R_{\e}^2),
\end{align}
where $c_2=\int_{\B} w\frac{\partial w}{\partial r}\frac{x^2}{r}$ and $c_2<0.$
\medskip

From (\ref{4.1}) to (\ref{4.8}), we get that $W_{\e,1,1}(\bf{q})$ can be represented as follows:
\begin{align}
\label{4.9}
W_{\e,1,1}(\q)=-c_1\xi\s\frac{\partial F(q)}{\partial q_{1,1}}+c_2\e^2R_{\e}\partial^2\mu(0)+h.o.t.
\end{align}
Using the asymptotic behaviour of the Green function $G_{\sigma}$, we have
\begin{align}
\label{4.10}
W_{\e,1,1}(\q)=&-c_1\xi\s\partial_{q_{1,1}}(G_{\s}(q_1,q_2)+G_{\s}(q_1,q_k))+c_2\e^2R_{\e}\partial^2\mu(0)+h.o.t.\nonumber\\
=&~c_1\xi\s R_{\sigma}^{-\frac12}e^{-R_{\s}}
\left(\frac{q_1-q_2}{|q_1-q_2|}
+\frac{q_1-q_k}{|q_1-q_k|}\right)+c_2\e^2R_{\e}\partial^2\mu(0)+h.o.t.\nonumber\\
=&~2c_1\xi\s R_{\sigma}^{-\frac12}e^{-R_{\s}}\sin\left(\frac{\pi}{k}\right)+c_2\e^2R_{\e}\partial^2\mu(0)+h.o.t.
\end{align}
We set the leading term of $W_{\e,1,1}(\q)$ by $\widehat W_{\e,1,1}(\q)$, we see that $\widehat W_{\e,1,1}(\q)$ depends only on $R_{\e}$ and the zero root $R_{\e,0}$ of $\widehat W_{\e,1,1}(\q)$ satisfies
\begin{equation}
\label{4.11}
\xi\left(2\s R_{\e,0}\sin\left(\frac{\pi}{k}\right)\right)^{-\frac32}e^{-2\s R_{\e,0}\sin\left(\frac{\pi}{k}\right)}+c_3D=0,
\end{equation}
where $c_3=\frac{c_2\partial^2\mu(0)}{4c_1(\sin\frac{\pi}{k})^2}$ is negative and
\begin{align*}
R_{\e,0}=\frac{1}{2\s\sin\left(\frac{\pi}{k}\right)}
\left(\log\frac{1}{D}
-\frac32\log\log\frac{1}{D}
-\log\frac{\xi}{c_3}
+O\left(\frac{\log\log\frac{1}{D}}{\log\frac{1}{D}}
\right)
\right)
\end{align*}
If $\xi^{-1}D$ is sufficiently small, then we easily get that equation (\ref{4.11}) admits a unique solution and it is nondegenerate. As a consequence, in the neighborhood of $R_{\e,0}$, we can find $\hat R_{\e,0}$ such that $W_{\e,1,1}=0$ and get the existence of (\ref{1.2}). Thus, we have proved Theorem \ref{th1.1}.
\vspace{1cm}

\section{Stability analysis I: Study of Large Eigenvalues}
 To prove Theorem 1.2, we consider the stability of the solution $(A_{\e},H_{\e})$ for (\ref{1.2}) which was given in Theorem 1.1.

Linearizing the system (GM) around the equilibrium states $(A_{\e},H_{\e})$, we obtain the following eigenvalue problem:
\begin{align}
\label{5.1}
\left\{\begin{array}{l}
\Delta_x\phi_{\e}-\mu(\e x)\phi_{\e}+2\frac{A_{\e}}{H_{\e}}\phi_{\e}-\frac{A_{\e}^2}{H_{\e}^2}\psi_{\e}=\lambda_{\e}\phi_{\e},\\[2mm]
\Delta_x\psi_{\e}-\sigma^2\psi_{\e}+2\xi A_{\e}\phi_{\e}=\tau\lambda_{\e}\sigma^2\psi_{\e},
\end{array}\right.
\end{align}
Here $\lambda_{\e}$ is some complex number and
$$\phi_{\e}\in H^2_N(\BB),\quad\psi_{\e}\in H^2_N(\BB).$$
In this section, we study the large eigenvalues, i.e., we assume that $|\lambda_{\e}|\geq c>0$ for $\e$ small.
The derivation of the a matrix characterising the small eigenvalues will be done in Appendix B since this study is quite technical.
Finally, in the next section, we discuss the small eigenvalues explicitly by considering these matrices. That part is central to understanding the
stability of spike clusters.

 If $\Re(\lambda_{\e})\leq-c,$ we are done. (Then $\lambda_{\e}$ is a stable large eigenvalue.) Therefore we may assume that {$\Re(\lambda_{\e})\geq-c$ and for a subsequence in $\e,D,$} we have $\lambda_{\e}\rightarrow\lambda_0\neq 0$. We shall derive the limiting eigenvalue problem of \eqref{5.1} as $\max\left(\frac{\e}{\sqrt{D}},D\log\frac{\sqrt{D}}{\e}\right)\rightarrow0$ which reduces to a system of NLEPs.

The key reference are Theorem \ref{th2.1} and Lemma \ref{le2.4}.

The second equation in (\ref{5.1}) is equivalent to
\begin{equation}
\label{5.2}
\Delta\psi_{\e}-\sigma^2(1+\tau\lambda_{\e})\psi_{\e}+2\xi_{\e}A_{\e}\phi_{\e}=0.
\end{equation}
We introduce the following:
\begin{equation*}
\sigma_{\lambda_{\e}}=\sigma\sqrt{1+\tau\lambda_{\e}},
\end{equation*}
where in $\sqrt{1+\tau\lambda_{\e}}$ we take the principal part of the square root. This means that the real part of $\sqrt{1+\tau\lambda_{\e}}$ is positive, which is possible because $\Re(1+\tau\lambda_{\e})\geq\frac12$.

Let us assume that $\|\phi_{\e}\|_{H^2(\BB)}=1$. We cut off $\phi_{\e}$ as follows:
$$\phi_{\e,j}(x)=\phi_{\e}(x)\chi_{\e,q_j}(x),$$
where the test function $\chi_{\varepsilon,q_j}(x)$ was introduced in \eqref{5.test}.

{From $\Re(\lambda_{\e})\geq-c$ and the exponential decay of $w$, we can derive from (\ref{5.1}) that}
\begin{align*}
\phi_{\e}=\sum_{j=1}^k\phi_{\e,j}+h.o.t.~\mathrm{in}~H^2_N(\BB).
\end{align*}
Since $\|\phi_{\e}\|_{H^2(\BB)}=1$, by taking a subsequence, we may also assume that $\phi_{\e,j}\rightarrow\phi_j$ in $H^2(\BB)$ as $\max\left(\frac{\e}{\sqrt{D}},D\log\frac{\sqrt{D}}{\e}\right)\rightarrow0$ for $j=1,2,\cdots,k.$ We have by (\ref{5.2})
\begin{align}
\label{5.3}
\psi_{\e}(x)=\xi\int_{B_{R/\varepsilon}}G_{\sigma_{\lambda_\e}}(x,z) A_{\e}(z)\phi_{\e}(z)dz.
\end{align}
At $x=q_i,~i=1,2,\cdots,k$, we calculate
\begin{align}
\label{5.4}
\psi_{\e}(q_i)=~&\xi\int_{B_{R/\varepsilon}}G_{\sigma_{\lambda_\e}}(q_i,z)\sum_{j=1}^k w_{j}(z)\phi_{\e,j}(y)dy+h.o.t.\nonumber\\
=~&\frac{1}{2\pi}\xi\log\frac{1}{\sigma_{\lambda_{\e}}}\int_{B_{R/\varepsilon}}w\phi+h.o.t.
\end{align}
Substituting the above equation in the first equation of (\ref{5.1}), taking the limit
\\
$\max\left(\frac{\e}{\sqrt{D}},D\log\frac{\sqrt{D}}{\e}\right)\rightarrow0$,
we get
\begin{align}
\label{5.5}
\Delta_x\phi_i-\phi_i+
2w\phi_i- \frac{2}{1+\tau\lambda_0}\,
\frac{\int_{\B} w\phi_i\mathrm{d}y}{\int_{\B} w^2\mathrm{d}y}w^2=\lambda_0\phi_i,~i=1,2\cdots,k,
\end{align}
where $\phi_i\in H^2(\B)$. Then we have

\begin{theorem}
\label{th5.1}
Let $\lambda_{\e}$ be an eigenvalue of (\ref{5.1}) such that $\Re(\lambda_{\e})>-c$ for some $c>0$.
\begin{enumerate}
  \item [(1)] Suppose that (for suitable sequence $\max\left(\frac{\e}{\sqrt{D}},D\log\frac{\sqrt{D}}{\e}\right)\rightarrow0$) we have $\lambda_{\e}\rightarrow\lambda_0\neq0$. Then $\lambda_0$ is an eigenvalue of the problem (NLEP) given in \eqref{5.5}.
  \item [(2)] Let $\lambda_0\neq0$ with $\Re(\lambda_0)>0$ be an eigenvalue of the problem (NLEP) given in \eqref{5.5}. Then for $\max\left(\frac{\e}{\sqrt{D}},D\log\frac{\sqrt{D}}{\e}\right)$ small enough, there is an eigenvalue $\lambda_{\e}$ of (\ref{5.1}) with $\lambda_{\e}\rightarrow\lambda_0$ as $\max\left(\frac{\e}{\sqrt{D}},D\log\frac{\sqrt{D}}{\e}\right)\rightarrow0.$
\end{enumerate}
\end{theorem}

\begin{proof}
(1) of Theorem \ref{th5.1} follows by asymptotic analysis similar to the one obtained in Appendix A.

To prove (2) of Theorem \ref{th5.1}, we follow a compactness argument of Dancer.
For the details we refer to Chapter 4 of \cite{ww-book}.

\end{proof}

We now study the stability of \eqref{5.1}, by Lemma \ref{le2.4}, for any nonzero eigenvalue $\lambda_0$ in (\ref{5.5}) we have
$$\Re(\lambda_0)\leq c_0<0~\mathrm{for~some}~c_0>0.$$
Thus, by Theorem \ref{th5.1}, for $\max\left(\frac{\e}{\sqrt{D}},D\log\frac{\sqrt{D}}{\e}\right)$ small enough, all nonzero large eigenvalues of (\ref{5.1}) have strictly negative real parts. More precisely, all eigenvalues $\lambda_{\e}$
 of (\ref{5.1}) for which $\lambda_{\e}\rightarrow\lambda_0\neq0$ holds, satisfy $\Re(\lambda_{\e})\leq-c<0.$
\medskip

In conclusion, we have finished studying the large eigenvalues (of order $O(1)$) and derived results on their stability properties. It remains to study the small eigenvalues (of order $o(1)$) which will be done in Appendix B and the next section.

\section{Stability Analysis III: The study of the matrix $M_{\mu}(\q)$}
In this section, we shall study the matrix $M_{\mu}(\q)$, i.e., the hessian matrix of the term
\begin{align}
\label{8.1}
{\prod({\q})}=
\sum_{i\neq j}\xi\frac{1}{(\sigma|q_i-q_j|)^{\frac12}}e^{-\sigma|q_i-q_j|}+c_3\sum_{i=1}^k\mu(\e q_i),
\end{align}
where $c_3=-\frac{c_2}{c_1}$ with $c_1,c_2$ given in (\ref{4.4}) and (\ref{4.8}) respectively.

Using
\begin{equation}
\label{distq}
|q_1-q_2|=\frac{1}{\sigma}
\left(\log\frac{1}{D}
-\frac{3}{2}\log\log\frac{1}{D}
-\log\frac{\xi}{c_3}
+O\left(\frac{\log\log\frac{1}{D}}{\log\frac{1}{D}}
\right)
\right),
\end{equation}
(see (\ref{4.11}),
it is not difficult to see that
in leading order
\begin{align*}
\frac{\partial^2}{\partial q_{1,i}\partial q_{1,j}}{\prod({\q})}
 \sim ~&\xi\sigma^{\frac{3}{2}}e^{-\sigma|q_1-q_2|}\frac{1}{|q_1-q_2|^\frac{5}{2}}(q_1-q_2)_i(q_1-q_2)_j,
\nonumber\\
&+\xi\sigma^{\frac{3}{2}}e^{-\sigma|q_1-q_k|}\frac{1}{|q_1-q_k|^\frac{5}{2}}(q_1-q_k)_i(q_1-q_k)_j,
\end{align*}
where $i,j=1,2.$ By rotational symmetry, it is enough to compute
$$\frac{\partial^2}{\partial q_{1,i}\partial q_{1,j}},\frac{\partial^2}{\partial q_{1,i}\partial q_{2,j}},\frac{\partial^2}{\partial q_{1,i}\partial q_{k,j}},~i,j=1,2.$$
Terms which depend on $q_m$ are obtained by suitable rotation of terms which contain $q_1$. By straightforward computation, we have
\begin{align*}
\frac{\partial {\prod({\q})}}{\partial q_{1,i}}=&-\xi\sigma^{\frac12}e^{-\sigma|q_1-q_2|}\frac{(q_1-q_2)_i}{|q_1-q_2|^{\frac32}}-
\xi\frac{1}{2\sigma^{\frac12}}e^{-\sigma|q_1-q_2|}\frac{(q_1-q_2)_i}{|q_1-q_2|^{\frac52}}\\
&-\xi\sigma^{\frac12}e^{-\sigma|q_1-q_k|}\frac{(q_1-q_k)_i}{|q_1-q_k|^{\frac32}}-
\xi\frac{1}{2\sigma^{\frac12}}e^{-\sigma|q_1-q_k|}\frac{(q_1-q_k)_i}{|q_1-q_k|^{\frac52}}+h.o.t.,
\end{align*}
and
\begin{align}
\label{8.2}
\frac{\partial^2{\prod({\q})}}{\partial q_{1,i}\partial q_{1,j}}=
~&\xi\sigma^{\frac32}e^{-\sigma|q_1-q_2|}\frac{(q_1-q_2)_i(q_1-q_2)_j}{|q_1-q_2|^{\frac52}}\nonumber\\
&+\xi\sigma^{\frac32}e^{-\sigma|q_1-q_k|}\frac{(q_1-q_k)_i(q_1-q_k)_j}{|q_1-q_2|^{\frac52}}+h.o.t.,
\end{align}
For the terms $\frac{\partial^2{\prod({\q})}}{\partial q_{1,i}\partial q_{2,j}},~i,j=1,2$, we note that
\begin{align*}
\frac{\partial{\prod({\q})}}{\partial q_{1,1}}=\xi\sigma^{\frac12}e^{-\sigma|q_2-q_1|}\frac{(q_2-q_1)_1}{|q_2-q_1|^{\frac32}}
+\xi\frac{1}{2\sigma^{\frac12}}e^{-\sigma|q_2-q_1|}\frac{(q_2-q_1)_1}{|q_2-q_1|^{\frac52}}+h.o.t.,
\end{align*}
and
\begin{align*}
\frac{\partial{\prod({\q})}}{\partial q_{1,2}}=\xi\sigma^{\frac12}e^{-\sigma|q_2-q_1|}\frac{(q_2-q_1)_2}{|q_2-q_1|^{\frac32}}+
\xi\frac{1}{2\sigma^{\frac12}}e^{-\sigma|q_2-q_1|}\frac{(q_2-q_1)_2}{|q_2-q_1|^{\frac52}}+h.o.t..
\end{align*}
Then, we get
\begin{align}
\label{8.3}
\frac{\partial^2{\prod({\q})}}{\partial q_{1,1}\partial q_{2,i}}=-\xi\sigma^{\frac32}e^{-\sigma|q_2-q_1|}\frac{(q_2-q_1)_1(q_2-q_1)_i}{|q_2-q_1|^{\frac52}}+h.o.t.,~i=1,2
\end{align}
and
\begin{align}
\label{8.4}
\frac{\partial^2{\prod({\q})}}{\partial q_{1,2}\partial q_{2,i}}=-\xi\sigma^{\frac32}e^{-\sigma|q_1-q_2|}\frac{(q_2-q_1)_2(q_2-q_1)_i}{|q_1-q_2|^{\frac52}}+h.o.t.,~i=1,2.
\end{align}
Similarly, for the terms $\frac{\partial^2{\prod({\q})}}{\partial q_{1,i}\partial q_{k,j}},~i,j=1,2$, we have
\begin{align}
\label{8.5}
\frac{\partial^2{\prod({\q})}}{\partial q_{1,1}\partial q_{k,i}}=-\xi\sigma^{\frac32}e^{-\sigma|q_k-q_1|}\frac{(q_k-q_1)_1(q_k-q_1)_i}{|q_k-q_1|^{\frac52}}+h.o.t.,~i=1,2
\end{align}
and
\begin{align}
\label{8.6}
\frac{\partial^2{\prod({\q})}}{\partial q_{1,2}\partial q_{k,i}}=-\xi\sigma^{\frac32}e^{-\sigma|q_k-q_1|}\frac{(q_k-q_1)_2(q_k-q_1)_i}{|q_k-q_1|^{\frac52}}+h.o.t.,~i=1,2.
\end{align}
We now compute these expression in a coordinate system of tangential and normal coordinates around each spike. We remark that these coordinates are the same as in \cite{dwy}.
The spike locations are given by
\[
q_j^0= (\frac{R_\sigma}{
\s\sin\left(\frac{\pi}{k}\right)}
\cos\theta_j,
\frac{R_\sigma}{
\s\sin\left(\frac{\pi}{k}\right)}
\sin\theta_j),\quad j=1,\ldots,k,
\]
where
\[\theta_j=\frac{(j-1)2\pi}{k}+\alpha\]
and $\alpha\in \mathbb{R}$. Note that the phase shift $\alpha$ appears in the problem due to the rotational invariance of $\mu=\mu(|y|)$ and we can choose $\alpha=0$.
Then in local coordinates we can write
\[
q_j=q_j^0+q_{j,1}\frac{q_j}{|q_j|}+q_{j,2}\frac{q_j^\perp}{|q_j^\perp|},\quad j=1,\ldots,k,
\]
{where $q_j$} is the radial (normal) vector and the tangential vector $q_j^\perp$ is obtained from $q_j$ by rotation of $\pi/2$ in anti-clockwise direction.

From (\ref{8.2}) to (\ref{8.6}), using the local coordinate frames and elementary trigonometry, the leading order of the matrix $M_{\mu}(\q)$ is
\begin{align}
\label{8.7}
M_{\mu}(\q)
=&~\xi\sigma^{\frac32}e^{-\sigma|q_1-q_2|}\frac{1}{|q_1-q_2|^{\frac52}}
\left[\begin{matrix}
(\sin\frac{\pi}{k})^2(A_1+4I)&\sin\frac{\pi}{k}\cos\frac{\pi}{k}A_2\\[3mm]
-\sin\frac{\pi}{k}\cos\frac{\pi}{k}A_2&-(\cos\frac{\pi}{k})^2A_1
\end{matrix}\right]+h.o.t.,
\end{align}
where
\begin{align*}
A_1=\left[\begin{matrix}
-2&1&0&\cdots&0&1\\
1&-2&1&\cdots&0&0\\
\vdots&\vdots&\vdots&\ddots&\vdots&\vdots\\
1&0&0&\cdots&1&-2
\end{matrix}
\right]\quad\mathrm{and}\quad
A_2=\left[\begin{matrix}
0&1&0&\cdots&0&-1\\
-1&0&1&\cdots&0&0\\
\vdots&\vdots&\vdots&\ddots&\vdots&\vdots\\
1&0&0&\cdots&-1&0
\end{matrix}
\right].
\end{align*}

Before analyzing the matrix in (\ref{8.7}), we need some basic facts about circulant matrices. We follow the presentation in \cite{dwy} and include this material here for completeness.
Denote the $k$-dimensional complex vector space and the ring of $k\times k$ complex matrices by $\mathbb{C}^k$ and $\mathbb{M}_k,$ respectively. Let ${\bf b}=(b_1,b_2,\cdots,b_k)\in\mathbb{C}^k$, we define a shift operator $S:\mathbb{C}^k\rightarrow\mathbb{C}^k$ by
\begin{align*}
S(b_1,b_2,\cdots,b_k)=(b_k,b_1,\cdots,b_{k-1}).
\end{align*}

\noindent {{\bf {Definition 6.1.}}} The circulant matrix $B=\mbox{circ}({\bf b})$ associated to the vector $${\bf b}=(b_1,b_2,\cdots,b_k)\in\mathbb{C}^k$$ is the $k\times k$ matrix whose $n$th row is $S^{n-1}{\bf b}$:
\begin{align*}
B=\left(\begin{matrix}
b_1&b_2&\cdots&b_{k-1}&b_k\\
b_k&b_1&\cdots&b_{k-2}&b_{k-1}\\
\vdots&\vdots&\ddots&\vdots&\vdots\\
b_3&b_4&\cdots&b_1&b_2\\
b_2&b_3&\cdots&b_k&b_1
\end{matrix}\right).
\end{align*}
We denote by $\mbox{circ}(k)\subset\mathbb{M}_k$ the set of all $k\times k$ complex circulant matrices.
\medskip

With this notation, both $A_1$ and $A_2$ are $k\times k$ circulant matrices. In fact,
\begin{align*}
A_1=\mbox{circ}\{(-2,1,0,\cdots,0,1)\}~\mathrm{and}~A_2=\mbox{circ}\{(0,1,0,\cdots,0,-1)\}.
\end{align*}
Let $\epsilon=e^{\frac{2\pi i}{k}}$ be a primitive $k$th root of unity, we define
\begin{align*}
X_l=\frac{1}{\sqrt{k}}\left(1,\epsilon^{l},\epsilon^{2l},\cdots,\epsilon^{(k-1)l}\right)^T\in\mathbb{C}^k, ~\mathrm{for}~l=0,\cdots,k-1,
\end{align*}
and
\begin{align*}
P_k=\left(\begin{matrix}
1&1&\cdots&1&1\\
1&\epsilon&\cdots&\epsilon^{k-2}&\epsilon^{k-1}\\
\vdots&\vdots&\ddots&\vdots&\vdots\\
1&\epsilon^{k-2}&\cdots&\epsilon^{(k-2)^2}&\epsilon^{(k-2)(k-1)}\\
1&\epsilon^{k-1}&\cdots&\epsilon^{(k-1)(k-2)}&\epsilon^{(k-1)^2}
\end{matrix}\right).
\end{align*}

For the circulant matrix $B=\mbox{circ}({\bf b})$, let
\begin{align*}
\lambda_l=b_1+b_2\epsilon^l+\cdots+b_k\epsilon^{(k-1)l},~\mathrm{for}{l=0,\cdots,k-1}.
\end{align*}
A simple computation shows that $BX_l=\lambda_lX_l$. Hence $\lambda_l$ is an eigenvalue of $B$ with normalised eigenvector $X_l$. Since $\{X_1,\cdots,X_k\}$ is a linearly independent set of vectors in $\mathbb{C}^k$, all of the eigenvalues of $B$ are given by $\lambda_l,~l=0,\cdots,k-1.$ By direct computation, the eigenvalues of $A_1$ are
\begin{align*}
\lambda_{1,l}=-2+\epsilon^l+\epsilon^{(k-1)l}=-4\sin^2\frac{l\pi}{k},~\mathrm{for}~l=0,\cdots,k-1,
\end{align*}
and the eigenvalues of $A_2$ are
\begin{align*}
\lambda_{2,l}=\epsilon^l-\epsilon^{(k-1)l}=2i\sin\frac{2l\pi}{k},~\mathrm{for}~l=0,\cdots,k-1.
\end{align*}
Let $\mbox{diag}(a_1,a_2,\cdots,a_k)$ denote the diagonal matrix with diagonal entries $a_1,a_2,\cdots,a_k$ and
$$\mathcal{M}=\left[\begin{matrix}
(\sin\frac{\pi}{k})^2(A_1+4I)&\sin\frac{\pi}{k}\cos\frac{\pi}{k}A_2\\[3mm]
-\sin\frac{\pi}{k}\cos\frac{\pi}{k}A_2&-(\cos\frac{\pi}{k})^2A_1
\end{matrix}\right].$$
From the above discussion for the circulant matrix, we have
\begin{align*}
P^{-1}\mathcal{M}P=&
\left[\begin{matrix}
P_k^{-1}&0\\0&P_k^{-1}
\end{matrix}\right]\left[\begin{matrix}
(\sin\frac{\pi}{k})^2(A_1+4I)&\sin\frac{\pi}{k}\cos\frac{\pi}{k}A_2\\
-\sin\frac{\pi}{k}\cos\frac{\pi}{k}A_2&-(\cos\frac{\pi}{k})^2A_1
\end{matrix}\right]
\left[\begin{matrix}
P_k&0\\0&P_k
\end{matrix}\right]\nonumber\\
=&\left[\begin{matrix}
4(\sin\frac{\pi}{k})^2(I-D_1)&i\sin\frac{2\pi}{k}D_2\\
-i\sin\frac{2\pi}{k}D_2&4(\cos\frac{\pi}{k})^2D_1
\end{matrix}
\right],
\end{align*}
where
\begin{align*}
D_1=\mbox{diag}\left(0,(\sin\frac{\pi}{k})^2,(\sin\frac{2\pi}{k})^2,\cdots,(\sin\frac{(k-1)\pi}{k})^2\right),
\end{align*}
and
\begin{align*}
D_2=\mbox{diag}\left(0,\sin\frac{2\pi}{k},\sin\frac{4\pi}{k},\cdots,\sin\frac{2(k-1)\pi}{k}\right).
\end{align*}
 Next we divide the matrix $P^{-1}\mathcal{M}P$ into $k$ two by two matrixes, where the $l-$th matrix ($l=0,1,\cdots,k-1$) is given by
\begin{align}
\label{8.8}
\left[\begin{matrix}
4(\sin\frac{\pi}{k})^2(\cos\frac{l\pi}{k})^2&i\sin\frac{2\pi}{k}\sin\frac{2l\pi}{k}\\[3mm]
-i\sin\frac{2\pi}{k}\sin\frac{2l\pi}{k}&4(\cos\frac{\pi}{k})^2(\sin\frac{l\pi}{k})^2
\end{matrix}\right].
\end{align}
It is easy to see that the determinant of the above matrix is $0$ and its trace is positive. Further, we see that the zero eigenvector of the above matrix is
\begin{equation}
\label{8.9}
(\cos\frac{\pi}{k}\sin\frac{l\pi}{k},
i\sin\frac{\pi}{k}\cos\frac{l\pi}{k})^T.
\end{equation}
\medskip

Since the leading order matrix $\mathcal{M}$ admits zero eigenvalues with geometric multiplicity $k$, we have to expand the matrix $M_{\mu}(\q)$ to the next order to determine if these small {eigenvalues} have positive or negative real part.

Before doing that, we point out an useful fact. Let us consider for example the term $\frac{\partial^2{\prod({\q})}}{\partial q_{2,1}\partial q_{1,1}}$. By direct computation we get
\begin{equation}
\begin{aligned}
\label{8.10}
\frac{\partial{\prod({\q})}}{\partial q_{1,1}}=~&\xi\sigma^{\frac12}e^{-\sigma|q_1-q_2|}\frac{1}{|q_2-q_1|^{\frac{3}{2}}}(q_2-q_1)_1\\
=~&\tilde{K}(|q_1-q_2|)(q_2-q_1)_1.
\end{aligned}
\end{equation}
Computing another derivative of $\tilde{K}(|q_1-q_2|)(q_2-q_1)_1$ with respect to $q_{2,1},$ we note that there are two types of terms:
\[\left[\frac{\partial\tilde{K}(|q_1-q_2|)}{\partial q_{2,1}}\right](q_2-q_1)_1(q_2-q_1)_1
\quad \mbox{ and } \quad
\tilde{K}(|q_1-q_2|)\frac{\partial(q_2-q_1)_1}{\partial q_{2,1}}.
\]
The first term is of the same symmetry class as the leading order term (i.e., the higher order term differs from the leading order term only by some small factor). Therefore, this term can be absorbed into the leading-order matrix $\mathcal{M}$.

However, the second term is different and it has to be taken into account. In fact, we will see that these type of terms can be used to resolve the stability problem. We can re-write the second term as follows:
\begin{align}
\label{8.11}
\tilde{K}(|q_1-q_2|)\frac{\partial(q_2-q_1)_1}{\partial q_{2,1}}=-\tilde{K}(|q_1-q_2|)\frac12\frac{\partial^2}{\partial q_{2,1}\partial q_{1,1}}|q_2-q_1|^2.
\end{align}
Hence, up to some factors it is enough for us to consider the terms $\frac12\frac{\partial^2}{\partial q_{2,j}\partial q_{1,i}}|q_1-q_2|^2,~i,j=1,2.$ These terms together with $c_3\e^2\mu''(0)$ are the next order terms in the matrix $M_{\mu}(\q)$.

Using the local coordinate frames of $q_1$ and $q_2$ to express Cartesian local coordinates $x_{i,j},\,i,j,=1,2$, we get
\[
x_{1,1}=q_{1,1},\quad
x_{1,2}=q_{1,2},
\]
\[
x_{2,1}=q_{2,1}\cos\frac{2\pi}{k}-q_{2,2}\sin\frac{2\pi}{k},\quad
x_{2,2}=q_{2,1}\sin\frac{2\pi}{k}+q_{2,2}\cos\frac{2\pi}{k}.
\]
This implies
{\allowdisplaybreaks
{\begin{align*}
|q_1-q_2|^2=
~&\left|(q_1^0+q_{1,1}\frac{q_1}{|q_1|}+q_{1,2}\frac{q_1^\perp}{|q_1^\perp|})
-(q_2^0+q_{2,1}\frac{q_2}{|q_2|}+q_{2,2}\frac{q_2^\perp}{|q_2^\perp|})\right|^2\nonumber\\
\nonumber\\
=~&(R_{\sigma}\cos\frac{\pi}{k}+x_{2,2}-x_{1,2})^2
+(R_\sigma \sin\frac{\pi}{k}- x_{2,1}+x_{1,1})^2
\nonumber\\
=~&R_{\sigma}^2+2R_{\sigma}
(\cos\frac{2\pi}{k}q_{2,2}+\sin\frac{2\pi}{k}q_{2,1}-q_{1,2})\cos\frac{\pi}{k}
\nonumber\\ & +
2R_{\sigma}
(-\cos\frac{2\pi}{k}q_{2,1}+\sin\frac{2\pi}{k}q_{2,2}+q_{1,1})\sin\frac{\pi}{k}
\nonumber\\ &
+(\cos\frac{2\pi}{k}q_{2,2}+\sin\frac{2\pi}{k}q_{2,1}-q_{1,2})^2
+(-\cos\frac{2\pi}{k}q_{2,1}+\sin\frac{2\pi}{k}q_{2,2}+q_{1,1})^2
\nonumber\\=~&R_{\sigma}^2+2R_{\sigma}
(\cos\frac{2\pi}{k}q_{2,2}+\sin\frac{2\pi}{k}q_{2,1}-q_{1,2})\cos\frac{\pi}{k}
\nonumber\\ & +
2R_{\sigma}
(-\cos\frac{2\pi}{k}q_{2,1}+\sin\frac{2\pi}{k}q_{2,2}+q_{1,1})\sin\frac{\pi}{k}
\nonumber\\
&+q_{1,1}^2+q_{1,2}^2+q_{2,1}^2+q_{2,2}^2
-2q_{1,1}q_{2,1}\cos\frac{2\pi}{k}
+2q_{1,1}q_{2,2}\sin\frac{2\pi}{k}
\nonumber\\ &
-2q_{1,1}q_{2,2}\cos\frac{2\pi}{k}
-2q_{1,2}q_{2,1}\sin\frac{2\pi}{k}.
\end{align*}}}

As a consequence, we have
\begin{align}
\label{8.12}
\frac{\partial^2|q_1-q_2|^2}{\partial^2q_{i,j}}=2,~i,j=1,2,
\quad
\frac{\partial^2|q_1-q_2|^2}{\partial q_{1,1}\partial q_{2,1}}
=\frac{\partial^2|q_1-q_2|^2}{\partial q_{1,2}\partial q_{2,2}}
=-2\cos\frac{2\pi}{k},
\end{align}
\begin{align}
\label{8.13}
\frac{\partial^2|q_1-q_2|^2}{\partial q_{1,1}\partial q_{2,2}}=2\sin\frac{2\pi}{k},\,
~\frac{\partial^2|q_1-q_2|^2}{\partial q_{1,2}\partial q_{2,1}}=-2\sin\frac{2\pi}{k},\,
\frac{\partial^2|q_1-q_2|^2}{\partial q_{1,1}\partial q_{1,2}}=
~\frac{\partial^2|q_1-q_2|^2}{\partial q_{2,1}\partial q_{2,2}}=0.
\end{align}
Similarly, in local coordinates $q_1,q_2$ we have
\begin{align*}
|q_1|^2=\left|q_1^0+q_{1,1}\frac{q_1}{|q_1|}+q_{1,2}\frac{q_1^\perp}{|q_1^\perp|}\right|^2
=R_{\s}^2+q_{1,1}^2+q_{1,2}^2+2q_{1,1}|R_{\s}|,
\end{align*}
where we used $q_{1}\cdot q_1^\perp=0.$ This implies
\begin{align}
\label{8.14}
\frac{\partial^2|q_1|^2}{\partial q_{1,1}^2}=\frac{\partial^2|q_1|^2}{\partial q_{1,2}^2}=2,~
\frac{\partial^2|q_1|^2}{\partial q_{1,1}\partial q_{1,2}}=0.
\end{align}
For the terms $c_3\e^2\mu''(0),$ from (\ref{4.10}) we derive that
\begin{align}
\label{8.15}
4(\sin\frac{\pi}{k})^2\tilde{K}(|q_1-q_2|)=c_3\e^2\mu''(0).
\end{align}
From the above discussion and (\ref{8.10})-(\ref{8.15}), expanding the matrix $M_{\mu}(\q)$ we get the following second order contribution:
\begin{align}
\label{8.16}
-\tilde{K}(|q_1-q_2|)\mathcal{M}_2
=-\tilde{K}(|q_1-q_2|)
\left[\begin{matrix}
\cos\frac{2\pi}{k}A_1&-\sin\frac{2\pi}{k}A_2\\[2mm]
\sin\frac{2\pi}{k}A_2&\cos\frac{2\pi}{k}A_1
\end{matrix}\right].
\end{align}
By using the matrix $P_k$, we diagonalize the matrix $\mathcal{M}_2$,
\begin{align*}
\left[\begin{matrix}
P_k^{-1}&0\\0&P_k^{-1}
\end{matrix}\right]
\mathcal{M}_2
\left[\begin{matrix}
P_k&0\\[2mm]
0&P_k
\end{matrix}\right]
=-\tilde{K}(|q_1-q_2|)
\left[\begin{matrix}
-4\cos\frac{2\pi}{k}D_1&-2i\sin\frac{2\pi}{k}D_2\\2i\sin\frac{2\pi}{k}D_2
&-4\cos\frac{2\pi}{k}D_1
\end{matrix}\right].
\end{align*}
From the discussion of the leading-order matrix $P^{-1}\mathcal{M}P$, we know that the vectors
$$v_{l,1}=(0,\cdots,\underbrace{\cos\frac{\pi}{k}\sin\frac{l\pi}{k}}_{l+1},0,\cdots,
\underbrace{i\sin\frac{\pi}{k}\cos\frac{l\pi}{k}}_{k+l+1},\cdots,0)^T,\quad
 (l=0,1,\cdots,k-1)
$$
are the eigenvectors with zero eigenvalues of the diagonal form.
To show the stability of the eigenvalues in the linear subspace spanned by these eigenvectors, we have to evaluate the bilinear form with respect to these eigenvectors and show that
\[\mu_l=
\frac{\langle (P^{-1}\mathcal{M}_2P)v_{l,1},v_{l,1}\rangle}{\langle v_{l,1},v_{l,1}\rangle}
\geq 0, \quad (l=0,1,\ldots,k-1).
\]
If $\langle (P^{-1}\mathcal{M}_2P)v_{l,1},v_{l,1}\rangle= 0$
some further study is needed.
We compute
\begin{align*}
\mu_l=\frac{\langle(P^{-1}\mathcal{M}_2P)v_{l,1},v_{l,1}\rangle}{\langle v_{l,1},v_{l,1}\rangle}
=&
-4\cos\frac{2\pi}{k}(\sin\frac{l\pi}{k})^2
\\& +
\frac{4\sin\frac{2\pi}{k}\sin\frac{2l\pi}{k}
\cos\frac{\pi}{k}\sin\frac{\pi}{k}\cos\frac{l\pi}{k}\sin\frac{l\pi}{k}}
{
(\cos\frac{\pi}{k})^2(\sin\frac{l\pi}{k})^2
+
(\sin\frac{\pi}{k})^2(\cos\frac{l\pi}{k})^2
}.
\end{align*}

Next we discuss when all eigenvalues are positive (linearly stable solution) or some eigenvalues are negative (linearly unstable solution).

For $l=0$ we have $\mu_l=0$. This eigenvalue and its eigenvector are connected to rotational invariance of solutions.

For $l=1$ we compute the numerator in the expression for $\mu_1$ as
\[-8\cos\frac{2\pi}{k}(\sin\frac{\pi}{k})^4
(\cos\frac{\pi}{k})^2
+16
(\sin\frac{\pi}{k})^4(\cos\frac{\pi}{k})^4
\]
\[
=8(\sin\frac{\pi}{k})^4(\cos\frac{\pi}{k})^2>0.
\]

For $l=k-1$ we compare with the case $l=1$. The terms $\sin\frac{2l\pi}{k}$ and $\cos\frac{l\pi}{k}$ change sign, the other terms are the same as for $l=1$. The result is the same as for $l=1$.
The eigenvalues $l=1$ and $l=k-1$ together with their eigenvectors correspond to translations and they are stable.

For $l=2$ we compute the numerator of $\mu_2$ as
\[
-4\cos\frac{2\pi}{k}(\sin\frac{2\pi}{k})^2
[(\cos\frac{\pi}{k})^2(\sin\frac{2\pi}{k})^2
+
(\sin\frac{\pi}{k})^2(\cos\frac{2\pi}{k})^2]
\]\[
+4\sin\frac{2\pi}{k}\sin\frac{4\pi}{k}
\cos\frac{\pi}{k}\sin\frac{\pi}{k}\cos\frac{2\pi}{k}\sin\frac{2\pi}{k}
\]\[=
-4\cos\frac{2\pi}{k}(\sin\frac{2\pi}{k})^2
[(\cos\frac{\pi}{k})^2(\sin\frac{2\pi}{k})^2
+
(\sin\frac{\pi}{k})^2(\cos\frac{2\pi}{k})^2
-(\sin\frac{2\pi}{k})^2((\cos\frac{\pi}{k})^2-(\sin\frac{\pi}{k})^2)]
\]\[
=-4\cos\frac{2\pi}{k}(\sin\frac{2\pi}{k})^2(\sin\frac{\pi}{k})^2.
\]
Thus $\mu_2>0$ for $k=3$, $\mu_2=0$ for $k=4$ and
$\mu_2<0$ for $k=5,\,6,\ldots$.

The eigenvalue for $l=2$ and $k=4$ is zero in the first two leading orders.
To decide if it possibly contributes to an instability, further expansions are required. This computation is beyond the scope of this paper. We expect that the eigenvalue will be stable and the cluster with 4 spikes on a regular polygon is linearly stable.

By the consideration of the eigenvalues, the clusters with 2 spikes or 3 spikes on a polygon are both linearly stable.

For $k=4$ we have $\mu_0=2$, $\mu_1=\mu_3=0$
and $\mu_2=-2$.

Next we consider clusters with spikes located on a regular polygon plus a spike in its centre.

\section{Cluster of spikes on a polygon with centre}
In this section, we sketch how the approach can be adapted to show existence and stability of a cluster for which the spikes are located at the vertices of a regular $k$-polygon with centre.

The spike positions are
\begin{align}
\label{3.1a}
Q_{\e}=\left\{{\bf q}=(q_1,\cdots,q_k,0)\mid q_i=\left(2\tilde{R}_{\e}\cos\frac{2(i-1)\pi}{k},
2\tilde{R}_{\e}\sin\frac{2(i-1)\pi}{k}\right)\right\},
\end{align}
where $\tilde{R}_{\e}$ is chosen such that
\begin{align}
\label{3.2a}
\frac{1}{C}\frac{\sqrt{D}}{\e}\log\left(\frac{1}{D\log\frac{\sqrt{D}}{\e}}\right)\leq \tilde{R}_{\e}\leq C\frac{\sqrt{D}}{\e}\log\left(\frac{1}{D\log\frac{\sqrt{D}}{\e}}\right)
\end{align}
for some constant $C$ independent of $\e$ and $D$.

To get the radius for the equilibrium position,
we compute
\[
\tilde{W}_{\e,1,1}({\bf q})=c_1\xi\sigma \tilde{R}_\sigma^{-1/2}e^{-\tilde{R}_\sigma}
+c_2 \e^2 \tilde{R}_\e \partial^2 \mu(0)
+\mbox{h.o.t.}=0,
\]
where $\tilde{R}_\sigma=\sigma \tilde{R}_\e$.
We get
\begin{align*}
\tilde{R}_{\e,0}=\frac{1}{\s}\left(\log\frac{1}{D}
-\frac{3}{2}\log\log\frac{1}{D}
-\log\frac{\xi}{c_3}
+O\left(\frac{\log\log\frac{1}{D}}{\log\frac{1}{D}}
\right)
\right).
\end{align*}

Due to symmetry
we also have $\tilde{W}_{\e,1,2}({\bf q})=0$
and
$\tilde{W}_{\e,k+1,1}({\bf q})=
\tilde{W}_{\e,k+1,2}({\bf q})=
0$.
From this we get the existence of a steady state of spikes located at the $k$ vertices of a polygon and its centre, where $k$ can be any natural number.

Next we consider the stability of this spike cluster steady state. We assume that $k\leq 5$. We take the same rotated coordinates as above around the vertices of the polygon. For the origin located in the centre of the polygon we keep Cartesian coordinates $x_1$ and $x_2$.

The matrix $\mathcal{\tilde{M}}_\mu(\q)$ is now given as follows:
\[
\mathcal{\tilde{M}}_\mu(\q)=
\xi\sigma^{\frac32}e^{-\sigma|q_1|}\frac{1}{|q_1|^{\frac52}}
\left[\begin{matrix}
\tilde{M}_1&\tilde{M}_2\\\tilde{M}_3&\tilde{M}_4
\end{matrix}\right]+ h.o.t.
\]
\[
=\xi\sigma^{\frac32}e^{-\sigma|q_1|}\frac{1}{|q_1|^{\frac52}}
\mathcal{\tilde{M}}
+ h.o.t.,
\]
where
\begin{align*}
\tilde{M}_1=\left[\begin{matrix}
1&0&0&\cdots&0&-1\\[2mm]
0&1&0&\cdots&0&-\cos\frac{2\pi}{k}\\[2mm]
0&0&1&\cdots&0&-\cos\frac{4\pi}{k}\\[2mm]
\vdots&\vdots&\vdots&\ddots&\vdots&\vdots\\[2mm]
-1 & -\cos\frac{2\pi}{k}&-\cos\frac{4\pi}{k}&\cdots&-\cos\frac{2(k-1)\pi}{k}&\frac{k}{2}
\end{matrix}\right],
\end{align*}
\begin{align*}
\tilde{M}_2=\left[
\begin{matrix}
0&0&0&\cdots&0&0\\[2mm]
0&0&0&\cdots&0&-\sin\frac{2\pi}{k}\\[2mm]
0&0&0&\cdots&0&-\sin\frac{4\pi}{k}\\[2mm]
\vdots&\vdots&\vdots&\ddots&\vdots&\vdots\\[2mm]
0&0&0&\cdots&0&0
\end{matrix}
\right],
\end{align*}
\begin{align*}
\tilde{M}_3=\left[\begin{matrix}
0&0&0&\cdots&0&0\\
0&0&0&\cdots&0&0\\
0&0&0&\cdots&0&0\\
\vdots&\vdots&\vdots&\ddots&\vdots&\vdots\\
0&-\sin\frac{2\pi}{k}&-\sin\frac{4\pi}{k}&\cdots&-\sin\frac{2(k-1)\pi}{k}&0
\end{matrix}\right],
\end{align*}
\begin{align*}
\tilde{M}_4=\left[\begin{matrix}
0&0&0&\cdots&0&0\\
0&0&0&\cdots&0&0\\
0&0&0&\cdots&0&0\\
\vdots&\vdots&\vdots&\ddots&\vdots&\vdots\\
0&0&0&\cdots&0&\frac{k}{2}
\end{matrix}\right].
\end{align*}
We multiply $\mathcal{\tilde{M}}$ from the right by the vector $\mathbf{a}=(a_1,a_2,\ldots,a_k,a_{k+1},0,0,\dots,0,a_{2k+2})^T$, i.e. we assume that the components $k+2,\,k+3,\ldots, 2k+1$ are all zero, and we also multiply $\mathcal{\tilde{M}}$ from the left by the transpose of this vector. Further, we set $a_{k+1}=\alpha,\,a_{2k+2}=\beta$, where $\alpha$ and $\beta$ are some real numbers. 
Then we get
\[
\mathbf{a}^T \mathcal{\tilde{M}}\mathbf{a}=
\sum_{l=1}^{k}
(a_l-\alpha \cos\frac{2(l-1)\pi}{k}-\beta\sin\frac{2(l-1)\pi}{k})^2.
\]
This means the matrix $\mathcal{\tilde{M}}$ is positive semi-definite if it is restricted to the components
$1,2,\ldots,k,$ $k+1,2k+2$.
The eigenvalue of any eigenvector in this class is always nonnegative. It is zero if and only if
\[
a_l=\alpha \cos\frac{2(l-1)\pi}{k}+\beta\sin\frac{2(l-1)\pi}{k}
\quad\mbox{ for }l=1,2,\ldots,k,
\]
where $\alpha$ and $\beta$
are some real numbers
which are independent of $l$.
These eigenvectors  have positive eigenvalues for the second-order part of the matrix. Similar to the computation for a polygon without centre it can be shown that there are positive contributions to the eigenvalues coming from the components $k+1$ and $2k+2$ which are related to the spike at the centre. Note that these eigenvectors correspond to translations.

In addition we have to study the eigenvalue of any eigenvector orthogonal to this class, i.e. for which the components $1,2,\ldots,k,k+1,2k+2$ are zero and the components $k+2, \cdots, 2k+1$ are arbitrary.
The leading-order matrix $\mathcal{M}$ for the cluster of the polygon without centre defined in Section 6 is the second-order contribution here. It is positive semi-definite in this class (since $A_1$ is positive semi-definite). It is strictly positive definite except for the eigenvector
$(0,0,\ldots,0,1,1,\ldots,1)^T$ which has zero eigenvalue (since $A_1$ has zero eigenvector $(1,1,\ldots,1)^T$). Note that this eigenvector corresponds to rotations.
Here the components $k+2,\ldots,2k+1$ for the polygon with central spike become the components $k+1,\ldots, 2k$  of the vector for the polygon without centre since the components $k+1$ and $2k+2$ are dropped.

These computations show that the eigenvalues of $\mathcal{\tilde{M}}$ are nonnegative and they are zero only for eigenvectors which correspond to the rotational invariance of the problem.
Together we get the stability of the cluster with spikes located at the vertices of a regular polygon with $k\leq 5$ vertices plus one spike at its centre.

\section{Discussion}
We have shown the existence of spike clusters located near a nondegenerate minimum point of the precursor gradient for the Gierer-Meinhardt system such that the spikes are located on regular polygons. We have proved that these solutions are stable for two or three spikes and unstable for five or more spikes.

We have considered the problem in the rotationally symmetric case. We have assumed that the precursor and the domain are both rotationally symmetric.

It will be interesting to extend these results to the case that the precursor and the domain are not rotationally symmetric. We are currently studying these effects using the approach in \cite{dwy}, where the existence of spiky patterns for the Schr\"odinger equation has been extended from the case of a rotationally symmetric potential to the general case.

If $\mu$ is not rotationally symmetric generically there will be certain possible orientations of the spike cluster and we expect to have stable and unstable equilibrium orientations.

If the domain is not a disk higher order terms coming from the regular part of the Green's function generically will determine the orientation of the spike cluster and we expect to have stable and unstable equilibrium orientations.

Because of the smallness of the inhibitor diffusivity we expect that generically the influence of $\mu$ will dominate that of the domain boundary.

 Further analysis is needed to resolve the stability issue for a 4-spike cluster (regular polygon with 4 vertices).
 Further calculations are required to show that the $(k+1)$-spike cluster for $k\geq 6$ (regular polygon with $k$ vertices plus a spike in the centre) is stable or unstable. The stability problem in this case requires some new analysis since the interaction between spikes is of a different type from the one considered in this paper. These issues are currently under investigation.

Whereas in one space dimension the spikes in a cluster are aligned with equal distance in leading order (although they differ in higher order) \cite{ww5},
in two space dimensions further spike configurations seem to be possible such as
concentric multiple polygons or positions close to regular polygons. Similar configurations have been studied in \cite{dkw}.

Biologically speaking, the precursor is the information retained from a previous stage of development and the patterns discovered in the reaction-diffusion system at the present will be able to determine the development in the future. The Gierer-Meinhardt system with precursor can be considered a minimal model to describe this behaviour. Generally one has to study larger systems which take into account other effects to make more reliable biological predictions. Therefore it will be interesting to study reaction-diffusion systems of three and more components and investigate the role which spike clusters play in such systems. One such system is a consumer chain model for which existence and stability of a clustered spiky patterns has been investigated by the first two authors \cite{ww55}. However, a more systematic approach will be needed to gain a better understanding of the role played by spike clusters in guiding biological development.

\vspace{1cm}
\section{Appendix A: Some auxiliary results for Liapunov-Schmidt reduction in section 3}

In this appendix we will prove some results which are needed for Liapunov-Schmidt reduction in section 3. In particular, will show that

We are going to show that the equation
\begin{align*}
\pi_{\e,\q}\circ S_{\e,\q}\left(\begin{matrix}
A_{\e,\q}+\phi_{\e,\q}\\
H_{\e,\q}+\psi_{\e,\q}
\end{matrix}\right)=0
\end{align*}
has a unique solution $\Sigma_{\e,\q}=\left(\begin{matrix}
\phi_{\e,\q}\\
\psi_{\e,\q}
\end{matrix}\right)\in\mathcal{K}_{\e,\q}^{\perp}$ if $\max\left(\frac{\e}{\sqrt{D}},D\log\frac{\sqrt{D}}{\e}\right)$ is small enough.

Recall that
\begin{align*}
*\label{3.19}
\mathcal{L}_{\e,\q}=\pi_{\e,\q}\circ\tilde{L}_{\e,\q}:\mathcal{K}_{\e,\q}^{\perp}\rightarrow\mathcal{C}_{\e,\q}^{\perp}.
\end{align*}

We will first show that this linear operator is uniformly invertible. Then we will use this to prove the existence of $\Sigma_{\e,\q}=\left(\begin{matrix}
\phi_{\e,\q}\\
\psi_{\e,\q}
\end{matrix}
\right).$

As a preparation, in the following two propositions we show the invertibility of {the} corresponding linearised operator $\mathcal{L}_{\e,\q}.$

\begin{proposition}
\label{pr3.1}
Let $\mathcal{L}_{\e,\q}$ be given in (\ref{3.19}). There exists a positive constant $\bar\delta$ such that for all
$\max\left(\frac{\e}{\sqrt{D}},D\log\frac{\sqrt{D}}{\e}\right)\in(0,\bar\delta)$, we can find a positive constant $C$ which is independent of $\e,D$ such that
\begin{equation}
\label{3.20}
\|\mathcal{L}_{\e,\q}\Sigma\|_{L^2(B_{R/\varepsilon})}
\geq C\|\Sigma\|_{H^2(B_{R/\varepsilon})}
\end{equation}
for arbitrary ${\q}\in Q_{\e},\Sigma\in\mathcal{K}_{\e,\q}^{\perp}.$ Further, the map $\mathcal{L}_{\e,\q}$ is surjective.
\end{proposition}

\noindent{\em Proof of Proposition \ref{pr3.1}.} Suppose (\ref{3.20}) is false. Then there exist sequences $\{\e_n\},\{D_n\}$, $\{\q^n\},$ $\{\phi^n\},\{\psi^n\}$ such that $\max\left(\frac{\e_n}{\sqrt{D_n}},D_n\log\frac{\sqrt{D_n}}{\e_n}\right)\rightarrow0$, $\q^n\in Q_{\e}$, $\phi^n=\phi_{\e_n,D_n,\q^n},n=1,2,\cdots$ and
\begin{equation}
\label{a.1}
\|\mathcal{L}_{\e_n,\q^n}\Sigma_n\|_{L^2(B_{R/\varepsilon})}\rightarrow 0~\mathrm{as}~n\rightarrow\infty,
\end{equation}
\begin{equation}
\label{a.2}
\|\phi^n\|_{H^2(B_{R/\varepsilon})}+\|\psi^n\|_{H^2(B_{R/\varepsilon})}=1,~n=1,2,\cdots.
\end{equation}
On the other hand, we note $\psi^n$ satisfies
\begin{align*}
\Delta\psi^n-\sigma^2\psi^n+2A_{\e_n,\q^n}\phi^n=0.
\end{align*}
It is easy to see from the above equation we get $\|\psi^n\|_{H^2(B_{R/\varepsilon})}\leq C\|\phi^n\|_{H^2(B_{R/\varepsilon})}$. Then we can assume $\|\phi^n\|_{H^2(B_{R/\varepsilon})}=1$. We define $\phi_{n,i},i=1,2,\cdots,k$ and $\phi_{n,k+1}$ as follows:
\begin{equation}
\label{a.3}
\phi_{n,i}(x)=\phi^n(x)\chi(\frac{x-q_i}{R_{\e}\sin\frac{\pi}{k}})~\mathrm{and}~
\phi_{n,k+1}(x)=\phi^n(x)-\sum_{i=1}^k\phi_{n,i}(x),
\end{equation}
where $\chi(x)$ is defined in \eqref{5.test}. Since for $i=1,2,\cdots,k$ each sequence $\{\phi_{n,i}\},~n=1,2,\cdots$ is bounded in $H^2_N(\B)$ it has a weak limit in $H^2_N(\B)$, and therefore also a strong limit in $L^2(\B)$ and $L^\infty(\B)$. Call these limit $\phi_i.$ Then $\Phi=(\phi_1,\cdots,\phi_k)^T$ solves the system $L\Phi=0$, where $L$ is given in (\ref{2.6}). By Lemma \ref{le2.3}, $\Phi\in Ker(L)=K_0\oplus\cdots\oplus K_0$. Since $\phi^n\perp K_{\e_n,\q_n}$, by taking $n\rightarrow\infty$ we get $\Phi\in Ker(L)^{\perp}$. Therefore, $\Phi=0.$

By elliptic estimates we have $\|\phi_{n,i}\|_{H^2(B_{R/\varepsilon})}\rightarrow0$ as $n\rightarrow\infty$ for $i=1,2,\cdots,k.$ Further, since $\phi_{n,k+1}\rightarrow\phi_{k+1}$ we get
\begin{equation*}
\Delta\phi_{k+1}-\phi_{k+1}=0~\mathrm{in}~\BB.
\end{equation*}
Therefore we conclude $\phi_{k+1}=0$ and $\|\phi_{n,k+1}\|_{H^2(B_{R/\varepsilon})}\rightarrow0$ as $n\rightarrow\infty$. This contradicts to $\|\phi^n\|_{H^2(B_{R/\varepsilon})}=1$.

To complete the proof of Proposition \ref{pr3.1} we just need to show that conjugate operator to $\mathcal{L}_{\e,\q}$ (denoted by $\mathcal{L}_{\e,\q}^*$) is injective from $\mathcal{K}_{\e,\q}^{\perp}$ to $\mathcal{C}_{\e,\q}^{\perp}$ and the proof for $\mathcal{L}_{\e,\q}^*$ follows almost the same process as for $\mathcal{L}_{\e,\q}$ and we omit it. Thus we finish the proof of Proposition \ref{pr3.1}. $\quad\quad\quad\square$

\medskip

Now we are in position to solve the equation
\begin{equation}
\label{3.21}
\pi_{\e,\q}\circ S_{\e,\q}\left(\begin{matrix}
A_{\e,\q}+\phi\\
H_{\e,\q}+\psi\end{matrix}\right)=0.
\end{equation}
Since $\mathcal{L}_{\e,\q}\mid_{\mathcal{K}_{\e,\q}^{\perp}}$ is invertible (call the inverse $\mathcal{L}_{\e,\q}^{-1}$), we can rewrite (\ref{3.21}) as
\begin{equation}
\label{3.22}
\Sigma=-(\mathcal{L}_{\e,\q}^{-1}\circ \pi_{\e,\q})
\left(S_{\e,\q}\left(\begin{matrix}
A_{\e,\q}\\
H_{\e,\q}\end{matrix}\right)\right)-(\mathcal{L}_{\e,\q}^{-1}\circ \pi_{\e,\q})(N_{\e,\q}(\Sigma))\equiv M_{\e,\q}(\Sigma),
\end{equation}
where
\begin{align*}
\Sigma=\left(\begin{matrix}\phi\\ \psi\end{matrix}\right),~
N_{\e,\q}(\Sigma)=S_{\e,\q}\left(\begin{matrix}A_{\e,\q}+\phi\\H_{\e,\q}+\psi\end{matrix}\right)
-S_{\e,\q}\left(\begin{matrix}A_{\e,\q}\\H_{\e,\q}\end{matrix}\right)
-S_{\e,\q}'\left(\begin{matrix}A_{\e,\q}\\H_{\e,\q}\end{matrix}\right)
\left[\begin{matrix}\phi\\ \psi\end{matrix}\right],
\end{align*}
and the operator $M_{\e,\q}$ is defined by (\ref{3.22}) for $\Sigma\in H^2_N(B_{R/\varepsilon})\times H^2_N(B_{R/\varepsilon})$. We are going to show the operator $M_{\e,\q}$ is a contraction map in
\begin{align*}
B_{\e,D}=\left\{\Sigma\in H^2_N(B_{R/\varepsilon})\times H^2_N(B_{R/\varepsilon})\mid \|\Sigma\|_{H^2(B_{R/\varepsilon})}<\eta\right\}
\end{align*}
provided $\max\left(\frac{\e}{\sqrt{D}},D\log\frac{\sqrt{D}}{\e}\right)$ is small enough. We have by Lemma \ref{le3.1} and Proposition \ref{pr3.1} that
\begin{align*}
\|M_{\e,\q}(\Sigma)\|_{H^2)N(\BB)}\leq&~
C\Big(\|\pi_{\e,\q}\circ N_{\e,\q}(\Sigma)\|_{L^2(B_{R/\varepsilon})}+\left\|\pi_{\e,\q}\circ S_{\e,\q}\left(\begin{matrix}A_{\e,\q}\\H_{\e,\q}\end{matrix}\right)\right\|_{L^2(B_{R/\varepsilon})}
\Big)\\
\leq&~ C(c(\eta)\eta+c_{\e,D}),
\end{align*}
where $C>0$ is independent of $\eta$, $c(\eta)\rightarrow0$ as $\eta\rightarrow0$ and $c_{\e,D}\rightarrow0$ as $$\max\left(\frac{\e}{\sqrt{D}},D\log\frac{\sqrt{D}}{\e}\right)\rightarrow0.$$
Similarly we can show
\begin{align*}
\left\|M_{\e,\q}(\Sigma)-M_{\e,\q}(\Sigma')\right\|_{H^2(B_{R/\varepsilon})}\leq Cc(\eta)\|\Sigma-\Sigma'\|_{H^2(B_{R/\varepsilon})}.
\end{align*}
If we choose $\eta$ sufficient small, then $M_{\e,\q}$ is a contraction map on $B_{\e,D}$. The existence of the fixed point $\Sigma_{\e,\q}$ together with an error estimate now follow from the contraction mapping principle. Moreover $\Sigma_{\e,\q}$ is a solution of (\ref{3.22}).
\medskip

{Thus we have proved}
\begin{lemma}
\label{le3.2}
There exists $\bar\delta>0$ such that for $$\max\left(\frac{\e}{\sqrt{D}},D\log\frac{\sqrt{D}}{\e}\right)\in(0,\bar\delta)$$
and $\q\in Q_{\e},$ {we can find a unique} $(\Phi_{\e,\q},\Psi_{\e,\q})\in\mathcal{K}_{\e,\q}^{\perp}$ satisfying
$$S_{\e,\q}\left(\begin{matrix}A_{\e,\q}+\Phi_{\e,\q}\\H_{\e,\q}+\Psi_{\e,\q}\end{matrix}\right)\in \mathcal{C}_{\e,\q}$$ and
\begin{equation}
\label{3.23}
\|(\Phi_{\e,\q},\Psi_{\e,\q})\|_{H^2(B_{R/\varepsilon})}\leq C\left(\frac{1}{\log\frac{\sqrt{D}}{\e}}+\e^2 R_{\e}^2\right).
\end{equation}
\end{lemma}

In the following, we need more refined estimates on $\Phi_{\e,\q}$. We recall that $S_1$ can be decomposed into the two parts $S_{1,1}$ and $S_{1,2}$, where $S_{1,1}$ is in leading order an odd function and $S_{1,2}$ is in leading order a radially symmetric function. Similarly, we can decompose $\Phi_{\e,\q}$:

\begin{lemma}
\label{le3.3}
Let $\Phi_{\e,\q}$ be defined in Lemma \ref{le3.2}. Then for $x=q_i+z,~|\s z|<\delta,$ we have
\begin{align}
\label{3.24}
\Phi_{\e,\q}=\Phi_{\e,\q,1}+\Phi_{\e,\q,2},
\end{align}
where $\Phi_{\e,\q,2}$ is a radially symmetric function in $z$ and
\begin{equation}
\label{3.25}
\|\Phi_{\e,\q,1}\|_{H^2(B_{R/\varepsilon})}=O\left(\e^2 R_{\e}+\s(\log\frac{1}{\s})^{-1}R_{\s}^{-\frac12}e^{-R_{\s}}+\s^2(\log\frac{1}{\s})^{-1}\right).
\end{equation}
\end{lemma}

\begin{proof}
Let $S[u]:=S_1(u,T[u]).$ We first solve
\begin{equation}
\label{3.26}
S\big[A_{\e,\q}+\Phi_{\e,\q,2}\big]-S\big[A_{\e,\q}\big]+\sum_{j=1}^kS_{1,2}(x-q_j)\in C_{\e,\q},
\end{equation}
for $\Phi_{\e,\q,2}\in K_{\e,\q}^{\perp}$.
Then we solve
\begin{align}
\label{3.27}
S\big[A_{\e,\q}+\Phi_{\e,\q,2}+\Phi_{\e,\q,1}\big]-S\big[A_{\e,\q}+\Phi_{\e,\q,2}\big]+\sum_{j=1}^kS_{1,1}(x-q_j)\in C_{\e,\q},
\end{align}
for $\Phi_{\e,\q,1}\in K^{\perp}_{\e,\q}.$ Using the same proof as in Lemma \ref{le3.2}, both equations (\ref{3.26}) and (\ref{3.27}) have unique solution for $\max\left(\frac{\sqrt{D}}{\e},D\log\frac{\sqrt{D}}{\e}\right)$ sufficiently small. By uniqueness, $\Phi_{\e,\q}=\Phi_{\e,\q,1}+\Phi_{\e,\q,2}$. It is easy to see that
$$\|S_{1,1}\|_{L^2(B_{R/\varepsilon})}=O\left(\e^2 R_{\e}+\s(\log\frac{1}{\s})^{-1}R_{\s}^{-\frac12}e^{-R_{\s}}+\s^2(\log\frac{1}{\s})^{-1}\right)$$
and $S_{1,2}\in C_{\e,\q}^{\perp}$. Then we can conclude that $\Phi_{\e,\q,1}$ and $\Phi_{\e,\q,2}$ have the required properties.
\end{proof}

\vspace{1cm}
\section{Appendix B: Stability Analysis II: Study of small eigenvalues}
In this section, we shall study the small eigenvalue for equation (\ref{5.1}). Namely, we assume $\lambda_{\e}\rightarrow0$ as $\max\left(\frac{\e}{\sqrt{D}},D\log\frac{\sqrt{D}}{\e}\right)\rightarrow0$. We shall show that the small eigenvalues are related to $\mu''(0)$ and the Green function.

Again let $(A_{\e},H_{\e})$ be the equilibrium state constructed for equation \eqref{1.2}. Let
$$A_{\e,j}=\chi_{\e,q_j}(x)A_{\e}(x),~j=1,2,\cdots,k,$$
where $\chi_{\e,q_j}$ is defined before \eqref{5.test}. Then it is easy to see that
\begin{equation}
\label{6.1}
A_{\e}=\sum_{j=1}^kA_{\e,j}+h.o.t.~\mathrm{in}~H^2_N(B_{R/\varepsilon}).
\end{equation}

In last section, we have derived the nonlocal eigenvalue (\ref{5.5}). Let us now set $\lambda_0=0$ in \eqref{5.5}, we have that
\begin{equation}
\label{6.2}
\Delta\phi_i-\phi_i+2w\phi_i-2\frac{\int_{\B}w\phi_i}{\int_{\B}w^2}w^2=0,
\end{equation}
which is equivalent to
\begin{align*}
L_0\left(\phi_i-2\frac{\int_{\B}w\phi_i}{\int_{\B}w^2}w\right)=0,~i=1,\cdots,k,
\end{align*}
where $L_0$ is defined in (\ref{2.7}). By Lemma \ref{le2.1}, we have
\begin{align}
\label{6.3}
\phi_i-2\frac{\int_{\B}w\phi_i}{\int_{\B}w^2}w\in \mathrm{Span}\left\{\frac{\partial w}{\partial x_j},~j=1,2\right\},~i=1,2,\cdots,k.
\end{align}
Multiplying \eqref{6.3} by $w$ and integrating over $\B$ and summing up, we have $\int_{\B}w\phi_i=0,$ and hence
\begin{align}
\label{6.4}
\phi_j\in K_0=\mathrm{Span}\left\{\frac{\partial w}{\partial x_j},~j=1,2\right\},~i=1,2,\cdots,k.
\end{align}
(\ref{6.4}) suggests that, at least formally, we should have
\begin{align}
\label{6.5}
\phi_{\e}\sim\sum_{j=1}^k\sum_{i=1}^2a^{\e}_{j,i}\frac{\partial w}{\partial x_i}(x-q_j),
\end{align}
where $a_{j,i}^{\e}$ are some constant coefficients.
\medskip

Next we find a good approximation of $\frac{\partial w}{\partial x_i}(x-q_j).$
Note that $A_{\e,j}(x)\sim w(x-q_j)$ in $H^2(\BB)$, and $A_{\e,j}$ satisfies
\begin{equation*}
\Delta A_{\e,j}-\mu(\e x)A_{\e,j}+\frac{A_{\e,j}}{H_{\e}^2}+h.o.t.=0.
\end{equation*}
Then we find $\frac{\partial A_{\e,j}}{\partial x_i}$ satisfies
\begin{align}
\label{6.6}
\Delta\frac{\partial A_{\e,j}}{\partial x_i}-\mu(\e x)\frac{\partial A_{\e,j}}{\partial x_i}+2\frac{A_{\e,j}}{H_{\e}}\frac{\partial A_{\e,j}}{\partial x_i}
-\frac{A_{\e,j}^2}{H^2_{\e}}\frac{\partial H_{\e}}{\partial x_i}-\e\mu'(\e x)\frac{x_i}{r}A_{\e,j}+h.o.t.=0,
\end{align}
and we have $\frac{\partial A_{\e,j}}{\partial x_i}=(1+o(1))\frac{\partial w}{\partial x_i}(x-q_j).$

We now decompose
\begin{align}
\label{6.7}
\phi_{\e}=\sum_{j=1}^k\sum_{i=1}^2a_{j,i}^{\e}\frac{\partial A_{\e,j}}{\partial x_i}+\phi_{\e}^{\perp}
\end{align}
with complex numbers $a_{j,i}^{\e}$, where
\begin{align}
\label{6.8}
\phi_{\e}^{\perp}\perp\tilde K_{\e,\q}:=\mathrm{Span}\left\{\frac{\partial A_{\e,j}}{\partial x_i}\mid j=1,2,\cdots,k,~i=1,2\right\}.
\end{align}
Our idea is to show that this is a good choice because the error $\phi^{\perp}_{\e}$ is small in a suitable norm and thus can be neglected.

For $\psi_{\e}$, we make the following decomposition according to $\phi_{\e}$
\begin{align}
\label{6.9}
\psi_{\e}=\sum_{j=1}^k\sum_{i=1}^2a_{j,i}^{\e}\psi_{\e,j,i}+\psi_{\e}^{\perp},
\end{align}
where $\psi_{\e,j,i}$ is the unique solution of the following problem,
\begin{equation}
\label{6.10}
\Delta\psi_{\e,j,i}-\sigma^2(1+\tau\lambda_{\e})\psi_{\e,j,i}+2\xi_{\e}A_{\e,j}\frac{\partial A_{\e,j}}{\partial x_i}=0~\mathrm{in}~B_{R/\varepsilon},
\end{equation}
and
\begin{equation}
\label{6.11}
\Delta\psi_{\e}^{\perp}-\sigma^2(1+\tau\lambda_{\e})\psi_{\e}^{\perp}+2\xi_{\e}A_{\e}\phi_{\e}^{\perp}=0
~\mathrm{in}~B_{R/\varepsilon}.
\end{equation}
Suppose that $\|\phi_{\e}\|_{H^2(B_{R/\varepsilon})}=1$. We have
$a_{j,i}^{\e}=\frac{\int_{B_{R/\varepsilon}}\phi_{\e}\frac{\partial A_{\e,j}}{\partial x_i}}{\int_{\B}(\frac{\partial w}{\partial x_1})^2},$
therefore, we get $|a_{j,i}^{\e}|\leq C.$
\medskip

Substituting the decomposition of $\phi_{\e}$ and $\psi_{\e}$ into the first equation in (\ref{5.1}), we have
\begin{equation}
\begin{aligned}
\label{6.12}
&\sum_{j=1}^k\sum_{i=1}^2a_{j,i}^{\e}\frac{A_{\e,j}^2}{H_{\e}^2}\Big(-\psi_{\e,j,i}+\frac{\partial H_{\e}}{\partial x_i}\Big)
+\e\sum_{j=1}^k\sum_{i=1}^2a_{j,i}^{\e}\mu_i(\e x)A_{\e,j}\\
&\quad +\Delta\phi_{\e}^{\perp}-\mu(\e x)\phi_{\e}^{\perp}+2\frac{A_{\e}}{H_{\e}}\phi_{\e}^{\perp}
-\frac{A_{\e}^2}{H_{\e}^2}\psi_{\e}^{\perp}-\lambda_{\e}\phi^{\perp}_{\e}+h.o.t.\\
&=\lambda_{\e}\sum_{j=1}^k\sum_{i=1}^2a_{j,i}^{\e}\frac{\partial A_{\e,j}}{\partial x_i},
\end{aligned}
\end{equation}
where $\mu_i=\frac{\partial\mu}{\partial y_i}.$ We set
\begin{align}
\label{6.13}
\mathcal{I}_1={\sum_{j=1}^k\sum_{i=1}^2a_{j,i}^{\e}\frac{A_{\e,j}^2}{H^2_{\e}}
\left(-\psi_{\e,j,i}+\frac{\partial H_{\e}}{\partial x_i}\right)}
+\e\sum_{j=1}^k\sum_{i=1}^2a_{j,i}^{\e}\mu_i(\e x)A_{\e,j},
\end{align}
and
\begin{align}
\label{6.14}
\mathcal{I}_2=\Delta\phi_{\e}^{\perp}-\mu(\e x)\phi_{\e}^{\perp}+2\frac{A_{\e}}{H_{\e}}\phi_{\e}^{\perp}
-\frac{A_{\e}^2}{H_{\e}^2}\psi_{\e}^{\perp}-\lambda_{\e}\phi_{\e}^{\perp}.
\end{align}
\medskip

Next, we shall first derive the estimate for $\phi_{\e}^\perp.$ Using (\ref{6.12}), since $\phi_{\e}^\perp\perp\tilde{K}_{\e,\q}$, then similar to the proof of Proposition \ref{pr3.1}, we have that
\begin{align}
\label{6.15}
\|\phi^\perp_{\e}\|_{H^2(B_{R/\varepsilon})}\leq C\|\mathcal{I}_1\|_{L^2(B_{R/\varepsilon})}.
\end{align}
So our aim it to estimate the term $\mathcal{I}_1$.
We note that
\begin{align*}
\frac{\partial H_{\e}}{\partial x_i}=~&\xi\int_{B_{R/\varepsilon}}\frac{\partial}{\partial x_i}G_{\sigma}(x,z)A^2_{\e}(z)\mathrm{d}z\\
=~&\xi\int_{B_{R/\varepsilon}}\frac{\partial}{\partial x_i}(K_{\sigma}(x,z)+H_{\sigma}(x,z))A^2_{\e}(z)\mathrm{d}z+h.o.t.
\end{align*}
and (for $\tau=0$)
\begin{align*}
\psi_{\e,j,i}(x)=~&2\xi\int_{B_{R/\varepsilon}}G_{\sigma}(x,z)A_{\e,j}\frac{\partial A_{\e,j}}{\partial z_i}\mathrm{d}z\nonumber\\
=~&\xi\int_{B_{R/\varepsilon}}(K_{\sigma}(x,z)+H_{\sigma}(x,z)+o(\sigma^2))\frac{\partial}{\partial z_i}A_{\e,j}^2,
\end{align*}
where $K_{\sigma}(x,z)$ and $H_{\s}(x,z)$ refer to the singular and regular part of the Green function $G_{\sigma}(x,z),$ respectively.
For $\tau>0$ the last formula
should use $G_{\s_{\lambda_{\e}}}$ instead of $G_{\s}$.  Since $\lambda_\e\to 0$ it can be shown by comparing the two Green functions that the error from this term does not contribute to the small eigenvalues in leading order. We omit the details.

Then from the above two formulas, we have
\begin{align*}
\frac{\partial H_{\e}}{\partial x_i}-\psi_{\e,j,i}(x)=
~&\xi\int_{B_{R/\varepsilon}}\frac{\partial}{\partial x_i}K_{\s}(x,z)A_{\e,j}^2(z)-K_{\s}(x,z)\frac{\partial}{\partial x_i}A_{\e,j}^2\mathrm{d}z\nonumber\\
~&+\xi\int_{B_{R/\varepsilon}}\frac{\partial}{\partial x_i}H_{\s}(x,z)A_{\e,j}^2(z)-H_{\s}(x,z)\frac{\partial}{\partial x_i}A_{\e,j}^2\mathrm{d}z\nonumber\\
~&+\xi\int_{B_{R/\varepsilon}}\frac{\partial}{\partial x_i}\sum_{l\neq j}G_{\s}(x,z)A_{\e,l}^2\mathrm{d}z+h.o.t..
\end{align*}
Using the fact that
{\begin{align}
\label{6.16}
\frac{\partial}{\partial x_i}K_{\s}(x,z)+\frac{\partial}{\partial z_i}K_{\s}(x,z)=h.o.t., ~\mathrm{for}~x\neq z,
\end{align}}
and integrating by parts, we get
\begin{align}
\label{6.17}
\frac{\partial H_{\e}}{\partial x_i}-\psi_{\e,j,i}(x)=\xi\int_{\B}w^2\left(\frac{\partial}{\partial x_i}F_j(x)+h.o.t.\right),
\end{align}
where $F_j(x)=H_{\s}(x,q_j)+\sum_{l\neq j}G_{\s}(x,q_l).$ Then from (\ref{6.13}), we get
\begin{align*}
\mathcal{I}_1=a_{j,i}^{\e}\left(\xi\frac{\partial F_j(x)}{\partial x_i}\int_{\B}w^2+\e\mu'(\e x)\frac{x_i}{r}w(x)+h.o.t.\right).
\end{align*}
From the proof of Theorem \ref{th1.1}, we observe that
\begin{align*}
\xi\int_{\BB}w^2\frac{\partial F_j(x)}{\partial x_i}w^2(x)+\e\mu'(\e x)\frac{x_i}{r}w(x)=o(\e^2R_{\e}).
\end{align*}
Hence, we have
\begin{align}
\label{6.18}
\|\mathcal{I}_1\|_{L^2(\BB}=o(\e^2R_{\e}\sum_{j=1}^k\sum_{i=1}^2|a_{j,i}^{\e}|)
\end{align}
and
\begin{align}
\label{6.19}
\|\phi_{\e}^\perp\|_{H^2(\BB)}\leq C\|\mathcal{I}_1\|_{L^(\BB)}=o(\e^2R_{\e}\sum_{j=1}^k\sum_{i=1}^2|a_{j,i}^{\e}|).
\end{align}
Using the equation for $\psi_{\e}^\perp$ and (\ref{6.19}), we obtain that
\begin{align}
\label{6.20}
\psi_{\e}^\perp(x)=o(\e^2 R_{\e}\sum_{j=1}^k\sum_{i=1}^2|a_{j,i}^{\e}|).
\end{align}
We calculate
{\begin{align}
\label{6.21}
\int_{\BB}\mathcal{I}_2\frac{\partial A_{\e,l}}{\partial x_m}\mathrm{d}z=~&
\int_{\BB}\left(\frac{A_{\e,l}^2}{H_{\e}^2}\left(\frac{\partial H_{\e}}{\partial x_m}\phi_{\e}^\perp
-\frac{\partial A_{\e,l}}{\partial x_m}\psi_{\e}^{\perp}\right)\right)-\lambda_{\e}\int_{\BB}\phi_{\e}^\perp\frac{\partial A_{\e,l}}{\partial x_m}\nonumber\\
=~&\int_{\BB}\frac{A_{\e,l}^2}{H_{\e}^2}\left(\frac{\partial H_{\e}}{\partial x_m}(q_l+x)-\frac{\partial H_{\e}}{\partial x_m}(q_l)\right)\phi_{\e}^\perp\nonumber\\
&+\int_{\BB}\frac{A_{\e,l}^2}{H_{\e}^2}\left(\frac{\partial H_{\e}}{\partial x_m}(q_l)\right)\phi_{\e}^\perp
-\lambda_{\e}\int_{\BB}\phi_{\e}^{\perp}\frac{\partial A_{\e,l}}{\partial x_m}\nonumber\\
&-\int_{\BB}\frac{A_{\e,l}^2}{H_{\e}^2}\frac{\partial A_{\e,l}}{\partial x_m}(\psi_{\e}^\perp(q_l+y)-\psi_{\e}^\perp(q_l))\nonumber\\
&-\lambda_{\e}\int_{\BB}\psi_{\e}^\perp(q_l)\frac{A_{\e,l}^2}{H_{\e}^2}\frac{\partial A_{\e,l}}{\partial x_m}\nonumber\\
=~&o\left(\s\e^2R_{\e}\sum_{j=1}^k\sum_{i=1}^2|a_{j,i}^{\e}|\right).
\end{align}}

Then we study the algebraic equation for $a_{j,i}^{\e}$. Multiplying both sides of (\ref{6.12}) by $\frac{\partial A_{\e,l}}{\partial x_m}$ and integrating over $\B$, we obtain
\begin{align*}
r.h.s.=~&\lambda_{\e}\sum_{j=1}^k\sum_{i=1}^2a_{j,i}^{\e}\int_{\BB}\frac{\partial A_{\e,j}}{\partial x_i}\frac{\partial A_{\e,l}}{\partial x_m}\nonumber\\
=~&\lambda_{\e}\sum_{j=1}^k\sum_{i=1}^2a_{j,i}^{\e}\delta_{jl}\delta_{im}\int_{\BB}\left(\frac{\partial w}{\partial z_1}\right)^2\mathrm{d}z(1+o(1))\nonumber\\
=~&\lambda_{\e}a_{l,m}^{\e}\int_{\BB}\left(\frac{\partial w}{\partial z_1}\right)^2\mathrm{d}z(1+o(1)).
\end{align*}
For the left hand side, we have
\begin{align}
\label{6.22}
l.h.s.=~&\sum_{j=1}^k\sum_{i=1}^2a_{j,i}^{\e}\int_{\BB}\frac{A_{\e}^2}{H_{\e}^2}\left(\frac{\partial H_{\e}}{\partial x_i}-\psi_{\e,j,i}\right)\frac{\partial A_{\e,l}}{\partial x_m}\nonumber\\
&+\sum_{j=1}^k\sum_{i=1}^2a_{j,i}^{\e}\int_{\BB}\e\mu'(\e x)\frac{x_i}{r}A_{\e,j}\frac{\partial A_{\e,l}}{\partial x_m}+\int_{\BB}\mathcal{I}_2\frac{\partial A_{\e,l}}{\partial x_m}.
\end{align}
Using (\ref{6.17}), we obtain
\begin{align}
\label{6.23}
l.h.s.=~&\sum_{j=1}^k\sum_{i=1}^2a_{j,i}^{\e}\int_{\BB}w^2\frac{\partial w}{\partial y_m}y_m\xi\int_{\BB}w^2(\frac{\partial^2}{\partial{q_{l,m}}\partial q_{j,i}}F(\q))\nonumber\\
&+\sum_{j=1}^k\sum_{i=1}^2a_{j,i}^{\e}\int_{\BB}\e^2\mu''(\e x)x_iA_{\e,j}\frac{\partial A_{\e,l}}{\partial x_m}\nonumber\\
&+o\left(\s\e^2R_{\e}\sum_{j=1}^k\sum_{i=1}^2|a_{j,i}^{\e}|\right).
\end{align}
From (\ref{6.22}) and (\ref{6.23}), we have
\begin{align}
\label{6.24}
l.h.s.=~&-c_1\xi\sum_{j=1}^k\sum_{i=1}^2a_{j,i}^{\e}\frac{\partial^2}{\partial q_{l,m}\partial q_{j,i}}F(\q)+c_2\e^2\sum_{j=1}^k\sum_{i=1}^2a_{j,i}^{\e}\delta_{jl}\delta_{mi}\mu''(\e q_i)+h.o.t.,
\end{align}
where $c_1,c_2$ are introduced in (\ref{4.4}) and (\ref{4.8}). Combining the $l.h.s.$ and $r.h.s.$, we have
\begin{align}
\label{6.25}
&\sum_{j=1}^k\sum_{i=1}^2a_{j,i}^{\e}\left(-c_1\xi\frac{\partial^2}{\partial q_{l,m}\partial q_{j,i}}F(\q)+c_2\e^2\delta_{jl}\delta_{mi}\mu''(\e q_j)\right)+h.o.t.\nonumber\\
&=~\lambda_{\e}a^{\e}_{l,m}\int_{\mathbb{R}^2}\left(\frac{\partial w}{\partial z_1}\right)^2\mathrm{d}z(1+o(1)).
\end{align}
From (\ref{6.25}), we see that the small eigenvalues with $\lambda_{\e}\rightarrow0$ and are related to the eigenvalue of the $2k\times 2k$ matrix $M_{\mu}(\q)$, where the $(j+ki,l+km)$-th component is the following
\begin{align*}
(M_{\mu}(\q))_{j+ki,l+km}=
-c_1\xi\frac{\partial^2}{\partial q_{l,m}\partial q_{j,i}}F(\q)+c_2\e^2\delta_{jl}\delta_{mi}\mu''(\e q_j),~1\leq j,l\leq k,1\leq i,m\leq2.
\end{align*}
In section 6
this matrix has been studied further and its eigenvectors and eigenvalues have been computed explicitly.

\end{document}